\documentclass[11pt,reqno,twoside,a4paper]{article}
\usepackage{mysty_arxiv}

\usepackage{wrapfig}
\usepackage{fullpage}
\title{Provable Phase Retrieval with Mirror Descent}

\author{Jean-Jacques Godeme\thanks{Normandie Univ, ENSICAEN, CNRS, GREYC, France. e-mail: \texttt{jean-jacques.godeme@unicaen.fr, Jalal.Fadili@ensicaen.fr}.} \and Jalal Fadili\footnotemark[1] \and Xavier Buet\thanks{Aix-Marseille Univ, CNRS, Centrale Marseille, Institut Fresnel, Marseille, France. \texttt{firstname.lastname@fresnel.fr}.} \and Myriam Zerrad\footnotemark[2] \and Michel Lequime\footnotemark[2] \and Claude Amra\footnotemark[2].}  

\date{}

\begin{document}

\maketitle
\begin{flushleft}\end{flushleft}
\begin{abstract}
In this paper, we consider the problem of phase retrieval, which consists of recovering an $n$-dimensional real vector from the magnitude of its $m$ linear measurements. We propose a mirror descent (or Bregman gradient descent) algorithm based on a wisely chosen Bregman divergence, hence allowing to remove the classical global Lipschitz continuity requirement on the gradient of the non-convex phase retrieval objective to be minimized. We apply the mirror descent for two random measurements: the \iid standard Gaussian and those obtained by multiple structured illuminations through Coded Diffraction Patterns (CDP). For the Gaussian case, we show that when the number of measurements $m$ is large enough, then with high probability, for almost all initializers, the  algorithm recovers the original vector up to a global sign change. For both measurements, the mirror descent exhibits a local linear convergence behaviour with a dimension-independent convergence rate. Our theoretical results are finally illustrated with various numerical experiments, including an application to the reconstruction of images in precision optics.
\end{abstract}


\section{Introduction}\label{sec:intro}
\subsection{Problem statement and motivations}
{In this work, we study phase retrieval which is an ill-posed inverse problem which consists in recovering a general signal from the intensity of its $m$ linear measurements, i.e., from phaseless observations. Historically, the first application of phase retrieval started with X-ray crystallography, and it now permeates many areas of imaging science with applications that include diffraction imaging, astronomical imaging, microscopy to name just a few; see \cite{shechtman_phase_2014,JaganathanReview16,luke_phase_2017} and references therein. One of the main applications motivating our work originates from precision in optics. Often components (e.g., interference filters) exhibit optical losses of order  $10^{-6}$ of the incident power. Super-polished surfaces are commonly used to circumvent this issue. Indeed, their roughness (responsible for losses by optical scattering) is very low compared to the illumination wavelength. Therefore, it is crucial to know how to characterize the roughness of polished surfaces. To do so, light scattering is ideal among the existing techniques because it is fast and non-invasive. The surface is illuminated with a laser source, and the diffusion is measured by moving a detector. Then the power spectral density of the surface topography can be directly measured thanks to the electromagnetic theory of light scattering; see \cite{amra_instantaneous_2018,Buet22} and references therein for a detailed description.

Our focus in this paper will be on the case of real signals. Formally, suppose $\avx \in \bbR^n$ is a signal and that we are given information about the squared modulus of the inner product between $\avx$ and $m$ sensing/measurement vectors $(a_r)_{r \in \tcb{\bbrac{m}}}$. The phase retrieval problem can be cast as:  
\begin{equation}\tag{GeneralPR}\label{GeneralPR}
\begin{cases}
\text{Recover $\avx\in \bbR^n$ from the measurements $y \in \bbR^m$} \\
y[r]=|\adj{a_r}\avx|^{2}, \quad r \in \tcb{\bbrac{m}} \eqdef \{1,\cdots,m\} 
\end{cases}
\end{equation}
where $y[r]$ is the $r$-th entry of the vector $y$.

Since $\avx$ is real-valued, the best one can hope is to ensure that $\avx$ is uniquely determined by $y$ up to a global sign. Phase retrieval is in fact an ill-posed inverse problem in general and is known to be NP-hard \cite{Sahinoglou91}. Thus, one of the major challenges is to design efficient recovery algorithms and find conditions on $m$ and $(a_r)_{r \in \tcb{\bbrac{m}}}$ which guarantee exact (up to a global sign change) and robust recovery.} 

%
%

\subsection{Prior work}
{
Our review here is by no means exhaustive and the interested reader should refer to the following references for comprehensive reviews \tcb{\cite{shechtman_phase_2014,JaganathanReview16,fannjiang_numerics_2020,vaswani_non-convex_2020}}.

\paragraph*{\textbf{Feasibility formulation of constrained phase retrieval}}
\tcb{In the one-dimensional case with Fourier measurements, it was shown by \cite{akutowicz_determination_1956,akutowicz_determination_1957,walther_question_1963} (see also \cite{Beinert15,Bendory2017} in the discrete case) that the phase retrieval problem without any a priori constraints lacks uniqueness (up to trivial ambiguities). This fundamental barrier does not apply in higher dimensions as pointed out in \cite{BruckSodin79} and shown in \cite{hayes_reconstruction_1982} for band-limited 2D signals, and uniqueness was shown to hold "generically" in \cite{barakat_algorithms_1985}.}

To circumvent this barrier, workarounds have been proposed that involve adding a constraint either implicitly or explicitly. Phase retrieval is then formulated as a feasibility problem, that is, as finding some point in the intersection of the set of points satisfying the constraints implied by the data measurements in \eqref{GeneralPR}, and the set of points satisfying constraints expressing some prior knowledge on the object to recover, such as support, band-limitedness, non-negativity, sparsity, etc. The Gerchberg and Saxton algorithm \cite{gerchberg_practical_1972}, proposed in  the early 70's in the optics literature, is an alternating projection algorithm to solve such a feasibility problem. Improved variants include Fienup's basic input-output and the hybrid input-output (HIO) \cite{crimmins_ambiguity_1981,fienup_phase_1982,crimmins_uniqueness_1983}. For the case of a support constraint alone, it has been identified by \cite{bauschke_phase_2002} that HIO corresponds to the now well-known Douglas-Rachford algorithm. Other fixed-point iterations based on projections that apply to constrained phase retrieval have also been proposed, such as the HPR scheme \cite{Bauschke04}, or RAAR \cite{Luke08} which is a relaxation of Douglas-Rachford. Thanks to a wealth of results in the variational analysis community, some convergence properties of these algorithms for the phase retrieval problem are now known. \tcb{One has to distinguish between the two important cases for feasibility problems: consistent and inconsistent. 

For consistent phase retrieval problems, it is known for instance that alternating projections is locally linearly convergent at points of intersection provided that the constraints do not intersect tangentially \cite{Drusvyatskiy15,Bauschke13,Noll16,Luke12,Lewis09,Lewis08}. Similar results are also known for Douglas-Rachford \cite{HesseLuke13,Phan16}. Global convergence guarantees are however only conjectured, and translating the non-tangential intersection into conditions on $m$ and $(a_r)_{r \in \tcb{\bbrac{m}}}$ remains open. 

For the inconsistent case, it was argued in \cite{luke_phase_2017} that almost any constraint, in particular compact support, will be inconsistent with the measurement process in optical phase retrieval problems. This means that the corresponding feasibility problems are inconsistent. In this even more challenging inconsistent phase retrieval setting, the only two works that we are aware of where local linear convergence of alternating projections and relaxed Douglas-Rachford to local best approximation points is established are \cite[Theorem~3.2 and Example~3.6]{Luke18} and \cite[Theorem~4.11 and Section~5]{Luke20}.}

\paragraph*{\textbf{Unconstrained phase retrieval}}
In the unconstrained setting of \eqref{GeneralPR}, the dominant approach in computational phase retrieval is to take more measurements to ensure well-posedness and improve the performance of phase retrieval algorithms. This idea of oversampling has been known for a while, and for instance in non-crystallographic modalities \cite{Miao99}. From a theoretical point of view, for the case where $(a_r)_{r \in \tcb{\bbrac{m}}}$ is a frame (redundant complete system), the authors in \cite{balan_signal_2006,balan_reconstruction_2016} derived various necessary and sufficient conditions for the uniqueness of the solution, as well as algebraic polynomial-time numerical algorithms valid for very specific choices of $(a_r)_{r \in \tcb{\bbrac{m}}}$. This approach is however of theoretical interest only and has drawbacks for instance that it requires specific types of measurements that cannot be realized in most
applications of interest.

A very different route consists in considering that the measurement vectors $(a_r)_{r \in \tcb{\bbrac{m}}}$ are sampled from an appropriate distribution, and then showing that when $m$ is on the order of $n$ (up to polylogarithmic factors),  then with high probability the original vector can be recovered exactly up to sign or phase change in the complex case, from the magnitude measurements in \eqref{GeneralPR}. This can can be done either through semidefinite convex relaxation or by directly attacking the non-convex formulation of the phase retrieval problem.

\paragraph*{\textbf{Convex relaxation}}
The key ingredient is to use a well-known trick turning a quadratic function on $\bbR^n$, such as in the data measurements in \eqref{GeneralPR}, into a linear function on the space of $n \times n$ matrices \cite{BenTal01,goemans_improved_1995}. Thus the recovery of a vector from quadratic measurements is lifted into that of recovering a rank-one Hermitian semidefinite positive (SDP) matrix from affine constraints, and the rank-one constraint is then relaxed into a convenient convex one. The two most popular methods in this line are PhaseLift \cite{candes_phase_2013-2} and PhaseCut \cite{waldspurger_phase_2015}. Both approaches are inspired by the matrix completion problem \cite{candes_matrix_2010} and they differ in the way factorization takes place. Exact and robust recovery with random Gaussian or CDP (Coded Diffraction Patterns) measurements using PhaseLift was established in \cite{candes_phaselift_2013,candes_solving_2014,candes_phase_2015}. For Gaussian measurements, \cite{candes_phaselift_2013} showed that exact recovery by PhaseLift holds for a sampling complexity bound $m \gtrsim n \log(n)$. This has been improved to a universal result with $m \gtrsim n$. Exact recovery by PhaseLift for CDP measurements was established in \cite{candes_phase_2015} for $m \gtrsim n \log^4(n)$, and has been improved to $m \gtrsim n \log^2(n)$ in \cite{gross_improved_2017}. While SDP based relaxations lead to solving tractable convex problems, the prospect of squaring the number of unknowns make them computationally prohibitive and impractical as $n$ increases. Since  then, more direct non-convex methods are again being proposed.


\paragraph*{\textbf{Nonconvex formulations}}  
The general strategy here is to use an initialization techniques that land one in a neighborhood of the optimal solution (up to global sign or phase change) where a usual iterative procedure from nonlinear programming with carefully chosen parameters can perform reliably. 

In \cite{Candes_WF_2015}, the authors use a spectral initialization and propose a gradient-descent type algorithm (Wirtinger flow) for solving the general complex phase retrieval problem by casting it as 
\begin{equation}\label{formulpro}
\min_{z\in\bbC^n} f(z) \eqdef \qsom{\paren{y[r]-|\adj{a_r}z|^{2}}^2} .
\end{equation}
For an appropriate (Wirtinger) gradient-descent step-size, they showed that with high probability, the scheme converges linearly to the true vector (up to a global phase change) for both Gaussian and CDP measurements provided that $m$ is on the order of $n$ up to polylogarithmic terms. 
A truncated version of the Wirtinger flow was proposed in \cite{chen_solving_2017} which uses careful selection rules providing a tighter initial guess, better descent directions and step-sizes, and thus enhanced performance. For Gaussian measurements, truncated Wirtinger flow was also shown to converge linearly to the correct solution and is robust to noise provided that $m \gtrsim n$. Other variants of Wirtinger flow possibly and/or other initializations were proposed in \cite{zhang_nonconvex_2017} and \cite{wang_solving_2017}, and were shown to enjoy similar guarantees in the noiseless case for Gaussian measurements. \tcb{The Polyak subgradient method to minimize $\frac{1}{m}\sum_{r=1}^m{|y[r]-|\transp{a_r}z|^{2}|}$ on $\R^n$ was proposed and analyzed in \cite{davis_nonsmooth_2020} for noiseless real phase retrieval with real isotropic sub-gaussian measurements. When properly initialized, its linear convergence was also shown for $m \gtrsim n$.}

An alternating minimization strategy, alternating between phase update and vector update, with a resampling-based initialization has been proposed in \cite{netrapalli_phase_2015} and was shown to enjoy noiseless exact recovery for $m \gtrsim n\log(n)^3$. A truncated version of the spectral initialization followed by alternating projection was also proposed in \cite{waldspurger_phase_2018} with exact recovery guarantees for Gaussian measurements under the sample complexity bound $m \gtrsim n$.

The authors in \cite{sun_geometric_2018} studied the landscape geometry of the nonconvex objective in \eqref{formulpro} for Gaussian measurements. They showed that for large enough number of measurements, i.e., $m \gtrsim n\log(n)^3$, there are no spurious local minimizers, all global minimizers are equal to the correct signal $\avx$, up to a global sign or phase, and the objective function has a negative directional curvature around each saddle point (that we coin strict saddles in our paper). This allowed them to describe and analyze a second-order trust-region algorithm to find a global minimizer without special initialization. The work of \cite{chen_gradient_2019} provides an analysis of global convergence properties of gradient descent for (real) Gaussian measurements and heavily relying on Gaussianity of the initialization. They required a sampling complexity bound $m \gtrsim n \mathrm{poly}\log(m)$ without making explicit the linear local convergence rate.
} 

\subsection{Contributions and relation to prior work}
{
In this paper, we consider the real\footnote{This is motivated by  main application in light scattering where the roughness of a surface to be recovered is real.} phase retrieval problem that we formulate \eqref{GeneralPR} as the minimization problem \eqref{formulepro}. Inspired by \cite{bolte_first_2017}, we propose a mirror descent (or Bregman gradient descent) algorithm with backtracking associated to a wisely chosen Bregman divergence, hence removing the classical global Lipschitz continuity requirement on the gradient of the nonconvex objective in \eqref{formulepro}.

In the deterministic case, we show that for almost all initializers, bounded iterates of our algorithm  converge to a critical point where the objective has no direction of negative curvature, i.e., a critical point which is not a strict saddle point. In addition, \tcb{provided that a local relative strong convexity property holds}, we also show that our mirror descent scheme exhibits a local linear convergence behaviour.  

In the case of \iid standard Gaussian measurements, provided that the the number $m$ of sensing vectors is large enough, it turns out that the iterates of our algorithm are bounded, \tcb{and that the set of critical points of the objective $f$ in \eqref{formulpro} is the union of $\ens{\pm \avx}$ and the set of strict saddle points.} This together with the above deterministic guarantees ensures that with high probability, for almost all initializers, our mirror descent recovers the original vector $\avx$ up to a global sign change, and exhibits a local linear convergence behaviour with a dimension-independent convergence rate. \tcb{Our results are far more general than those of \cite{chen_gradient_2019} as we require for instance a smaller sampling complexity bound and we assume any random initialization provided that it is drawn from a distribution that has a density \wrt the Lebesgue measure, \ie the Gaussian nature of initialization in \cite{chen_gradient_2019} is irrelevant in our context.}

\tcb{For both CDP and Gaussian measurements, we show that one can afford a smaller sampling complexity bound but at the price of using an appropriate spectral initialization procedure to find an initial guess near a solution before applying our scheme. Starting from this initial guess, mirror descent then converges linearly to the true vector up to a global sign change with a dimension-independent convergence rate. This is in contrast with the Wirtinger flow \cite{Candes_WF_2015} which also requires spectral initialization and whose local convergence rate degrades with the dimension, though the latter aspect has been improved in the truncated Wirtinger flow \cite{chen_solving_2017}.  The Polyak subgradient method \cite{davis_nonsmooth_2020} initialized with a spectral method provably converges linearly with isotropic sub-gaussian measurements under a sample complexity bound similar to ours. However, no analysis is known for the CDP measurement model. Observe also that the Polyak subgradient algorithm requires the knowledge of the minimal value of the phase retrieval objective. This is obviously $0$ for the noiseless case but is unknown in the noisy one. In terms of computational complexity, mirror descent involves solving the mirror step (see Proposition~\ref{prop:mirrorstep}) which amounts to computing the unique real positive root of a third order polynomial and then multiplying it by the entry vector. This costs $O(n)$ operations. Overall, the computational complexity of mirror descent is similar to that of other first-order methods such as the Wirtinger flow or the Polyak subgradient algorithm.}

Though we focus on Gaussian measurements when establishing the global recovery properties of our mirror descent algorithm, our theory extends to the situation where the $a_r$'s are \iid sub-Gaussian random vectors. The case where $a_r$'s are a drawn form the CDP model is, however, far more challenging. One of the main difficulties is that several of our arguments rely on uniform bounds, for instance on the Hessian, that need to hold simultaneously for all vectors $x \in \bbR^n$ with high probability. But the CDP model bears much less randomness to exploit for establishing such bounds with reasonable sampling complexity bounds. Whether this is possible or not is an open problem that we leave to future research.

\subsection{Paper organization}
The rest of the paper is organized as follows. In Section~\ref{sec:deter}, we describe the mirror descent algorithm with backtracking and establish its global and local convergence guarantees in the deterministic case. We then turn to the case of random measurements in Section~\ref{sec:rand} where we provide sample complexity bounds for the deterministic guarantees to hold with high probability. Section~\ref{sec:numexp} is devoted to the numerical experiments. The proofs of technical results are collected in the appendix.

\section{Deterministic Phase Retrieval}\label{sec:deter}
\paragraph*{\textbf{Notations}}
We denote $\pscal{\cdot,\cdot}$ the scalar product and $\normm{\cdot}$ the corresponding norm. $B(x,r)$ is the corresponding ball of \tcb{radius $r$} centered at $x$ and $\mathbb{S}^{n-1}$ is the corresponding unit sphere. For $m \in \N^*$, we use the shorthand notation $\tcb{\bbrac{m}}=\{1,\ldots,m\}$.  The $i$-th entry of a vector $x$ is denoted $x[i]$. Given a matrix $M$, $\transp{M}$ is its transpose and $\adj{M}$ is its adjoint (transpose conjugate). Let $\lambda_{\min}(M)$ and $\lambda_{\max}(M)$ be respectively the smallest and the largest eigenvalues of $M$. For two real symmetric matrices $M$ and $N$, $M \slon N$ if $M-N$ is positive semidefinite.
$\inte$ is the interior of a set. We denote by $\Gamma_0(\bbR^n)$ the class of proper lower semicontinuous convex function. $\dom(f)$ is the domain of the function $f$. \tcb{$f^*$ denotes the Legendre-Fenchel conjugate of $f$.} Recall that the set of critical points of $f \in C^1(\bbR^n)$ is $\crit (f) = \enscond{x \in \bbR^n}{\nabla f(x)=0}$.

Let us denote the set of true vectors by $\xoverline{\calX}=\{\pm\avx\}$. For any vector $x\in\bbR^n$, the distance to the set of true vectors is  
\begin{equation}
\dist(x,\xoverline{\calX})\eqdef \min\pa{\normm{x-\avx},\normm{x+\avx}} .
\end{equation}
We will also use the shorthand notation: $A$ is the $m \times n$ matrix with $\adj{a_r}$'s as its rows.

\tcb{
\begin{remark}
Our limitation of the set of true solutions to $\{\pm\avx\}$ may appear restrictive since even for real vectors, the equivalence class is much larger than what we are allowing. First, note that our deterministic convergence results in Theorem~\ref{Maintheo} apply at any global minimizer. Moreover, our restriction will be justified in the oversampling regime with random measurements. For instance, for Gaussian measurements, only $\{\pm\avx\}$ are provably global minimizers for large enough number of measurements. Moreover, for the two types of random measurements in Section~\ref{sec:rand}, spectral initialization also provides an initialization which is real and provably lies in the neighborhood of $\{\pm\avx\}$.
\end{remark}
}

\subsection{Bregman toolbox}
For any $\phi:\bbR^n \to ]-\infty,+\infty]$ such that $\phi \in \Gamma_0(\bbR^n) \cap C^1(\inte(\dom(\phi)))$, we define a proximity measure associated with $\phi$.
\begin{definition}\label{Brgdf}\textbf{(Bregman divergence)} 
The Bregman divergence associated with $\phi$ is defined as $D_\phi:\bbR^n\times \bbR^n\rightarrow ]-\infty,+\infty]$:
\begin{equation}\label{Bregmdef}
D_{\phi}(x,u)\eqdef 
\begin{cases} 
\phi(x) - \phi(u) + \pscal{\nabla \phi(u),x - u} &\si ~ (x,u)\in \dom(\phi)\times \inte(\dom(\phi)),\\
 +\infty & \odwz. 
\end{cases}
\end{equation}
\end{definition} 
The classical euclidean distance is generated by the energy entropy $\phi(x)=\frac{1}{2}\normm{x}^2$. More examples of entropies and associated Bregman divergences can be found in \cite{bolte_first_2017,Teboulle18}. Clearly, $D_\phi$ is not a distance (it is not symmetric in general for example). 

We now collect some of the properties of the Bregman divergence \tcb{that will be useful} in our context. See \cite{teboulle_entropic_1992,chen_convergence_1993} and \cite[Proposition~2.10]{bauschke_dykstras_2000} for the last claim.
\begin{proposition}\textbf{(Properties of the Bregman divergence)}\label{pp:bregman}
\begin{enumerate} [label=(\roman*)]
\setlength{\itemindent}{0.42cm}
\item $D_\phi$ is nonnegative if and only if $\phi$ is convex. If in addition $\phi$ is strictly convex, $D_\phi$ vanishes if and only if its arguments are equal. \label{pp:bregman1}
\item \label{pp:bregman2}
Linear additivity: for any $\alpha,\beta \in \bbR$ and any functions $\phi_1$ and $\phi_2$ we have
\begin{equation}\label{linear}
D_{\alpha \phi_1+\beta \phi_2}(x,u)=\alpha D_{\phi_1}(x,u) + \beta D_{\phi_2}(x,u),
\end{equation}
for all $(x,u)\in \left(\dom \phi_1\cap\dom \phi_2\right)^2$ such that both $\phi_1$ and $\phi_2$ are differentiable at $u$.
\item \label{pp:bregman3}
The three-point identity: for any $x \in \dom(\phi)$ and $u,z\in \inte(\dom(\phi))$, we have
\begin{equation}\label{3poin}
D_{\phi}(x,z)-D_{\phi}(x,u)-D_{\phi}(u,z)=\pscal{\nabla \phi(u)-\nabla \phi(z),x-u},
\end{equation}
\item \label{pp:bregman4}
Suppose that $\phi$ is also $C^2(\inte(\dom(\phi)))$ and $\nabla^2 \phi(x)$ is positive definite for any $x \in \inte(\dom(\phi))$. Then for every convex compact subset $\Omega\subset\inte(\dom(\phi))$, there exists $0<\theta_{\Omega}\leq\Theta_{\Omega}<+\infty$ such that for all $x,u \in \Omega$,
\begin{equation}\label{striconv}
\frac{\theta_{\Omega}}{2}\normm{x-u}^2\leq D_{\phi}(x,u)\leq\frac{\Theta_{\Omega}}{2}\normm{x-u}^2.
\end{equation}
\end{enumerate} 
\end{proposition}

We are now ready to extend the gradient Lipschitz continuity property to the Bregman setting, that we coin relative smoothness. The notion of relative smoothness is key to the analysis of differentiable but not Lipschitz-smooth optimization problems. The earliest reference to this notion can be found in an economics paper \cite{birnbaum2011distributed} where it is used to address a problem in game theory involving fisher markets. Later on it was developed in \cite{bauschke_descent_2016,bolte_first_2017} and then in \cite{lu2018relatively}, although first coined relative smoothness in \cite{lu2018relatively}.

\begin{definition}\label{smoothadaptable}\textbf{($L-$relative smoothness)}
Let $\phi \in \Gamma_0(\bbR^n) \cap C^1(\inte(\dom(\phi)))$, and $g$ be a proper and lower semicontinuous function such that $\dom(\phi) \subset \dom(g)$ and $g \in C^1(\inte(\dom(\phi)))$. $g$ is called \tcb{$L-$smooth relative to} $\phi$ on $\inte(\dom(\phi))$ if there exists $L>0$ such that $L\phi-g$ is convex on $\inte(\dom(\phi))$, \ie 
\begin{equation}\label{eq:smoothadaptable}
D_g(x,u) \leq LD_{\phi}(x,u) \qforallq (x,u) \in \dom(\phi) \times \inte(\dom(\phi)) .
\end{equation}
\end{definition}

When $\phi$ is the energy entropy, \ie $\phi=\frac{1}{2}\normm{\cdot}^2$, one recovers the standard descent lemma implied by Lipschitz continuity of the gradient of $g$. \\
  
In an analogous way, we also extend the standard local strong convexity property to a relative version \wrt to an entropy or kernel $\phi$.
\tcb{
\begin{definition}\label{relativeconvex}\textbf{(Local relative strong convexity)}
Let $\phi \in \Gamma_0(\bbR^n) \cap C^1(\inte(\dom(\phi)))$, and $g$ be a proper and lower semicontinuous function such that $\dom(\phi) \subset \dom(g)$ and $g \in C^1(\inte(\dom(\phi)))$. Let $\C$ be a non-empty subset of $\dom(\phi)$. For $\sigma > 0$, we say that $g$ is $\sigma$-strongly convex on $\C$ \tcb{relative} to $\phi$ if 
\begin{equation}
D_g(x,u)\geq \sigma D_{\phi}(x,u) \qforallq x \in \C \tandt u \in \C \cap \inte(\dom(\phi)) .
\end{equation}
\end{definition}
The idea of global (\ie $\C=\dom(\phi)$) relative strong convexity has already been used in the literature, see \eg  \cite[Proposition~4.1]{Teboulle18} and \cite[Definition~3.3]{Bauschke19}. Its local version was first proposed in \cite{Silveti22}. When $\phi$ is the energy entropy (\ie $\phi=\frac{1}{2}\normm{\cdot}^2$), one recovers the standard definition of (local/global) strong convexity. Relation of global relative strong convexity to gradient  dominated inequalities, which is an essential ingredient to prove global linear convergence of mirror descent, was studied in \cite[Lemma~3.3]{Bauschke19}.}  

\subsection{Phase retrieval minimization problem}
In this work, we cast \eqref{GeneralPR} as solving the following optimization problem
\begin{equation}\label{formulepro}
\min_{x\in\mathbb{R}^n}\left\{ f(x)\eqdef\qsom{\loss{y[r]}{|(Ax)[r]|^2}}\right\}.
\end{equation}
Observe that $f \in C^2(\bbR^n)$ but is obviously non-convex. It is also clear that $\nabla f$ is not Lipschitz continuous. This is the main motivation behind considering the framework of Bregman gradient descent. As we will see shortly, $f$ has a relative smoothness property (see Definition~\ref{smoothadaptable} above) with respect to a well-chosen entropy function. In turn, relative smoothness will prove crucial for establishing descent properties of Bregman gradient descent, also known as, mirror descent.

Following \cite{bolte_first_2017}, let us consider the following kernel or entropy function   
\begin{align}\label{KGDpr}
\psi(x)=\frac{1}{4}\normm{x}^4+\frac{1}{2}\normm{x}^2 . 
\end{align}

\tcb{
Recall that a function $\phi \in \Gamma_0(\R^n)$ is Legendre if it is strictly convex and differentiable on $\inte(\dom(\phi)) \neq \emptyset$, with $\normm{\nabla \phi(x_k)} \to +\infty$ for each sequence $\seq{x_k} \subset \inte(\dom(\phi))$ converging to a boundary point of $\dom(\phi)$.
\begin{proposition}\label{prop:psi}
$\psi$ enjoys the following properties:
\begin{enumerate}[label=(\roman*)]\setlength{\itemindent}{0.4cm}
\item $\psi \in C^2(\bbR^n)$, is 1-strongly convex and Legendre.
\item $\nabla \psi$ is Lipschitz over bounded subsets of $\bbR^n$. 
\item $\nabla \psi$ is a bijection from $\bbR^n$ to $\bbR^n$, and its inverse is $\nabla \psi^*$.
\end{enumerate}
\end{proposition}
The first two claims are easy to show. The last one follows from \cite[Theorem~26.5]{rockafellar_convex_1970}.
}

It turns out that the objective $f$ in \eqref{formulepro} is smooth \tcb{relative} to the entropy $\psi$ defined in \eqref{KGDpr} on the whole space $\bbR^n$. This is stated in the following result whose proof is provided in Appendix~\ref{PrTsmad}.
\begin{lemma}\label{Tsmad}
Let $f$ and $\psi$ as defined in \eqref{formulepro} and \eqref{KGDpr} respectively. $f$ is \tcb{$L$-smooth relative} to $\psi$ on $\bbR^n$ for any $L\geq \som{3\normm{a_r}^4}$. 
\end{lemma}

This estimate of of the modulus of relative smoothness $L$ in Lemma~\ref{Tsmad} is rather crude but has the advantage to not depend on the measurements $y$. A far sharper estimate will be provided in the case where the sensing vectors are random; see Section~\ref{sec:rand}.


\subsection{Mirror descent with backtracking}
We recall the following mapping closely related to the Bregman gradient descent. For all $x\in \bbR^n$ and any step-size $\gamma >0$, 
\begin{equation}\label{DProx}
T_{\gamma}(x)\eqdef\argmin\limits_{u \in\bbR^n}\left\{\pscal{\nabla f(x),u-x}+\frac{1}{\gamma}D_{\psi}(u,x) \right\}.
\end{equation}
The pair $(f,\psi)$ defined in \eqref{formulepro}-\eqref{KGDpr} satisfies \cite[Assumptions $A,B,C,D$]{bolte_first_2017} (in fact $\psi$ is even strongly convex \tcb{in our case}). Therefore, it is straightforward to see that $T_{\gamma}$ is a well-defined and single-valued on $\bbR^n$; see \cite[Lemma~3.1]{bolte_first_2017}. \tcb{Moreover, by virtue of Proposition~\ref{prop:psi}, letting $x^+=T_{\gamma}(x)$, the first order optimality condition for $\eqref{DProx}$ reads}
\begin{align}\label{MDalgo}
x^{+}=F(x) \eqdef \nabla\psi^{-1}\paren{\nabla\psi(x)-\gamma\nabla f(x)}=\nabla\psi^*\paren{\nabla\psi(x)-\gamma\nabla f(x)} .
\end{align}

Our mirror descent (or Bregman gradient descent) scheme with backtracking is summarized in Algorithm~\ref{alg:MDBT}.

\begin{algorithm}[H]
 \caption{Mirror Descent for Phase Retrieval}
 \label{alg:MDBT}
    \textbf{Parameters:} $L_0 = L$ (see Lemma~\ref{Tsmad}), $\kappa \in ]0,1[$, $\xi\geq1$ . \\
    \textbf{Initialization:} $x_0\in\bbR^n$\;
    \For{$k=0,1,\ldots$}{
    \Repeat{$D_{f}(\xkp,\xk) > L_{k}D_{\psi}(\xkp,\xk)$}{
    $\lk \leftarrow \lk/\xi$, 
    $\gak=\frac{1-\kappa}{\lk}$ \\
    $\xkp=F(\xk)=\nabla\psi^{*}\paren{\nabla\psi(\xk)-\gak\nabla f(\xk)}$ \\
    }
    $\lk \leftarrow \xi\lk$, $\gak=\frac{1-\kappa}{\lk}$ \\
    $\xkp=F(\xk)$.
    }
\end{algorithm}    
    
Observe that Algorithm~\ref{alg:MDBT} cannot be trapped in the second loop thanks to Lemma~\ref{Tsmad}. The version without backtracking is recovered by setting $\xi=1$ and using constant step-size verifying $\gamma \in ]0,1/L[$ where $L$ is the global relative smoothness coefficient. Backtracking for an inertial version of the Bregman proximal gradient algorithm was used in \cite{mukkamala_convex-concave_2020}. 

It remains now to compute the mirror step. This amounts to finding a root of a third-order polynomial. 
\begin{proposition}\textbf{(Mirror step computation)}{\cite[Proposition~5.1]{bolte_first_2017}}\label{prop:mirrorstep}
Let $x\in\bbR^n$ and $p_{\gamma}(x)=\nabla\psi(x)-\gamma\nabla f(x)$. Then computing \eqref{MDalgo} amounts to 
\begin{equation}
x^+=t^*p_{\gamma}(x), 
\end{equation}
where  $t^*$ is the unique real positive root of $t^3\normm{p_{\gamma}(x)}^2+t-1=0$.
\end{proposition}

\subsection{Deterministic recovery guarantees by mirror descent}
We pause to recall two notions that will be important in our convergence result.
\begin{definition}\label{AssumptionE}\textbf{($f$-attentive neighborhood)}
A point $u\in\bbR^n$ belongs to an $f$-attentive neighborhood of $x \in \bbR^n$, if there exist $\delta > 0$ and $\mu>0$ such that $u \in B(x,\delta)$ and $f(x)<f(u)<f(x)+\mu$.
\end{definition}

\begin{definition}\label{defa}\textbf{(Strict saddle points)}
A point $\xpa\in\crit(f)$ is a strict saddle point of $f$ if \linebreak $\lambda_{\min}(\nabla^2f(\xpa))< 0$. The set of strict saddle points of $f$ is denoted $\strisad(f)$.    
\end{definition}

We are now ready to state our main convergence result.
{
\begin{theorem}\label{Maintheo}
Let $\seq{\xk}$ be a bounded sequence generated by Algorithm~\ref{alg:MDBT} for the phase retrieval problem \eqref{GeneralPR}. Then,
    \begin{enumerate}[label=(\roman*)]
    \setlength{\itemindent}{0.45cm}
    \item \label{point-i} the sequence $\seq{f(\xk)}$ is non-increasing,
    \item \label{point-ii} the sequence $\seq{\xk}$ has a finite length and converges to a point in $\crit(f)$.
    \item  \label{point-iii} Let $r > 0$. Assume that the initial point $\xo$ is in the $f$-attentive neighborhood of $\xsol \in \Argmin(f) \neq \emptyset$, \ie $\exists \delta \in ]0,r[$ and $\mu > 0$ such that $\xo\in B(\xsol,\delta)$ and $f(\xo) \in ]0,\mu[$, then
    \begin{enumerate}[label=(\alph*)]
     \item \label{point-iii-a}  $\forall k\in \bbN,\xk \in B(\xsol,r)$, and $\xk$ converges to a global minimizer of $f$.
     \item \label{point-iii-b} Besides, if $\exists \rho>0$ such that $f$ is $\sigma$-strongly convex \tcb{on $B(\xsol,\rho)$ relative to $\psi$}, with $r\leq\frac{\rho}{\max\pa{\sqrt{\Theta(\rho)},1}}$, where we recall $\Theta(\rho)$ from Proposition~\ref{pp:bregman}\ref{pp:bregman4}, then $\forall k \in \bbN$
     \tcb{
     \begin{align}\label{loclinear}
      \normm{\xk - \xsol}^2 \leq \Ppa{\prod_{i=0}^{k-1}\frac{1-\sigma\gamma_i}{1+\sigma\gamma_i\Theta(\rho)^{-1}}}\rho^2 \to 0 .
     \end{align}
     }
    \end{enumerate}
    \item \label{point-iv} If $L_k = L$, then for Lebesgue almost all initializers $\xo$, the sequence $\seq{\xk}$ converges to an element in $\crit(f)\bsl\strisad(f).$
    \end{enumerate}   
\end{theorem}
    See Section~\ref{PrMaintheo} for the proof.
}
    {   
    \begin{remark}\label{rem:deterministic} {\ }
    \begin{itemize}
     \setlength{\itemindent}{0.2cm}
      \item A standard assumption that automatically guarantees the boundedness of the sequence $\seq{\xk}$, hence its convergence to a critical point, is coercivity of $f$. Since the latter is a composition of a coercive function (a positive quartic function) and the linear operator $A$ (recall that its rows are the $a_r^*$'s), coercivity of $f$ amounts to injectivity of $A$. This is exactly what we will show in the random case when $m$ is large enough.
      \item  It is clear that $\Argmin(f)\neq \emptyset$ since $\overline{\calX} \subset \Argmin(f)$ and the claim \ref{point-iii} applies at $\pm\avx$ in which case one has exact recovery up to a global sign.
      \item A close inspection at the proof of Proposition~\ref{pp:bregman}\ref{pp:bregman4} shows that \linebreak$\Theta(\rho) = \sup_{x \in B(\avx,\rho)} \normm{\nabla^2 \psi(x)}$ does the job. In view of \eqref{hessent}, it is easy to see that $\Theta(\rho) \leq 6\normm{\avx}^2+6\rho^2+1$.
      \item Claim \ref{point-iii} shows local linear convergence of $\xk$ to $\xsol$. Indeed, $\sigma \leq L_k$ for any $k$, and thus $1-\sigma \gamma_k \in ]\kappa,1[$.
     \item  Clearly, claim \ref{point-iv} states that when the initial point is selected according to a distribution which has a density \wrt the Lebesgue measure, then the sequence $(\xk)_{k \in \bbN}$ converges to a point that avoids strict saddle points of $f$. This is a consequence of the centre stable manifold theorem applied to our mirror descent algorithm. 
     \item When it will come to the phase retrieval problem from random measurements (see forthcoming section), in order to prove local linear convergence, the key argument will be to show that for a sufficient number of measurements, then \whp $f$ is \tcb{strongly convex around $\pm \avx$ relative to $\psi$}.
     

\end{itemize}
\end{remark}
}

\section{Random Phase Retrieval via Mirror Descent}\label{sec:rand}
\subsection{Framework}\label{framework}
Throughout the paper, we will work under two random measurement models:
\begin{enumerate}[label=(\arabic*)]
\item \label{framework-1}The sensing vectors are drawn \iid following a (real) standard Gaussian distribution. We can then rewrite the observation data as 
\begin{equation}\label{Noisy-PR}
y[r]=|\transp{a_r}\avx|^2, \quad r\in\tcb{\bbrac{m}},
\end{equation}
where $(a_r)_{r \in \tcb{\bbrac{m}}}$ are {\iid} $\calN(0,1)$.

\item \label{framework-2} The Coded Diffraction Patterns (CDP) model, as considered for instance in \cite{candes_phase_2015}. The idea is to modulate the signal before diffraction in the case of the Fourier transform measurements. The observation model is then 
\begin{eqnarray}\label{PRCDP}
y=\Ppa{|\calF(D_p\avx)[j]|^{2}}_{j,p}=\Ppa{\left|\sum_{\ell=0}^{n-1}\avx_{\ell}d_p[\ell] e^{-i\frac{2\pi j\ell}{n}}\right|^2}_{j,p}.
\end{eqnarray}
where $j\in\{0,\ldots,n-1\}$ and $p\in\{0,\ldots,P-1\}$, $D_p$ is a real diagonal matrix with the modulation pattern $d_p$ on its diagonal, and $\calF$ is the discrete Fourier transform. $P$ is the number of coded patterns/masks and the total number of measurements is then $m=nP$. The  modulation patterns  $(d_p)_{p \in [P]}$ are \iid copies of the same random vector $d$ satisfying the following assumption:
\begin{assumption}\label{assum:CDP}{\ }
 \begin{enumerate}[label=(A.\arabic*)]
  \item $d$ is symmetric and $\exists M>0$ such that $|d|\leq M$.
  \item Moments conditions: $\esp{d}=0$ and  $\esp{d^4}=2\esp{d^2}^2$. Without loss of generality, we assume $\esp{d^2}=1$. 
 \end{enumerate} 
\end{assumption}
For example, we can take ternary random variables with values in $\ens{-1, 0, 1}$ with probabilities $\ens{\frac{1}{4},\frac{1}{2},\frac{1}{4}}$. We refer to \cite{candes_phase_2015} for other modulation patterns. 
\end{enumerate}
\medskip

When the number of measurements is large enough for both measurements models, we will be able to establish local convergence properties of Algorithm~\ref{alg:MDBT} provided it is initialized with a good guess. For this, we use a spectral initialization method; see for instance \cite{Candes_WF_2015,chen_solving_2017,netrapalli_phase_2015,zhang_nonconvex_2017,wang_solving_2017,waldspurger_phase_2018}. The procedure consists of taking $\xo$ as the leading  eigenvector of a specific matrix as described in Algorithm~\ref{alg:algoSPIniit}.

\begin{algorithm}[H]\caption{Spectral Initialization.}
\label{alg:algoSPIniit}
\KwIn{$y[r], r=1,\ldots,m$.}
\KwOut{$\xo$}
Set $\lambda^2=n\frac{\sum_ry[r]}{\sum_r\normm{a_r}^2}$\;\BlankLine
Take $\xo$ the top eigenvector of $Y=\som{y[r]a_r\adj{a_r}}$ normalized to $\normm{\xo}=\lambda$.
\label{Spini}
\end{algorithm}

\tcb{
\begin{remark}
Assuming random measurements models and using probabilistic arguments to get sample complexity bounds and understand fundamental limits of phase retrieval (and other inverse problems) is an established technique in the applied mathematics literature. Of course, we are aware that this might not always be realistic from an application perspective as it may sometimes involve changing the data measurements to fit the theory. Nonetheless, for the application we have in mind (precision in optics), the CDP measurement model seems reasonable. This is the subject of an ongoing work.
\end{remark}
}

We are now ready to state our main results for each measurement model.

{ 
\subsection{Gaussian measurements}
Before stating our result, we consider the following events which will be helpful in our proofs. For this, we fix $\vrho \in ]0,1[$ and $\lambda \in ]0,1[$.
\begin{itemize}
\item The event 
\begin{equation}\label{eq:crit_char_G}
\calE_{\rm strictsad} = \ensB{\crit(f)=\overline{\calX}\cup\strisad(f)} 
\end{equation}
means that the set of critical points of the function $f$ is reduced to $\{\pm \avx\}$ and the set of strict saddle points. 
 
\item The event 
\begin{equation} \label{eq:uniconcen_G}
\calE_{\rm conH} = \left\{\forall x\in\bbR^n,\quad\normm{\nabla^2f(x)-\esp{\nabla^2f(x)}} \leq \vrho\para{\normm{x}^2+\normm{\avx}^2/3}\right\}  
\end{equation}
captures the deviation of the Hessian of $f$ around its expectation.

\item The event
\begin{equation}\label{eq:injectivity_G}
\calE_{\rm inj} =\left\{\forall x\in\bbR^n,\quad \paren{1-\vrho}\normm{x}^2\leq \frac{1}{m}\normm{A x}^2\right\}
\end{equation}
represents injectivity of the measurement matrix $A$.

\item $\calE_{\rm smad}$ is the event on which the function $f$ is $L$-smooth \tcb{relative to} $\psi$ in the sense of Definition~\ref{smoothadaptable}, with $L=3+\vrho\max\pa{\normm{\avx}^2/3,1}$. 

\item $\calE_{\rm scvx}$ is the event on which $f$ is $\sigma$-strongly convex \tcb{on $B(\xoverline{\calX},\rho)$ relative to} $\psi$ in the sense of Definition~\ref{relativeconvex}, with $\sigma=(\lambda\min\pa{\normm{\avx}^2,1}-\vrho\max\pa{\normm{\avx}^2/3,1})$ and $\rho=\frac{1-\lambda}{\sqrt{3}}\normm{\avx}$. 


\item We end up by denoting
\begin{equation}\label{eq:e1_conv_G}
\calE_{\rm conv}=\calE_{\rm strictsad}\cap\calE_{\rm conH}\cap\calE_{\rm inj}\cap\calE_{\rm smad}\cap\calE_{\rm scvx} .
\end{equation}
\end{itemize}

Our main result for Gaussian measurements is the following. 
\begin{theorem}\label{thm:Gauss}
Fix $\lambda \in ]0,1[$ and $\vrho \in]0,\lambda\min\pa{\normm{\avx}^2,1}/(2\max\pa{\normm{\avx}^2/3,1})[$. Let $\seq{\xk}$ be the sequence generated by Algorithm~\ref{alg:MDBT}.
\begin{enumerate}[label=(\roman*)]
\setlength{\itemindent}{0.35cm}
\item \label{thm:Gauss-i} If the number of measurements $m$ is large enough, \ie $m\geq C(\vrho)n\log^3(n)$, then for almost all initializers $x_0$ of Algorithm~\ref{alg:MDBT} used with constant step-size $\gamma_k \equiv \gamma = \frac{1-\kappa}{3+\vrho\max\pa{\normm{\avx}^2/3,1}}$, for any $\kappa \in ]0,1[$, we have
\[
\dist(\xk,\overline{\calX}) \to 0 ,
\] 
and $\exists K \geq 0,$ large enough such that $\forall k \geq K,$
\begin{align}\label{conv_lin_G}
\dist^2(\xk,\overline{\calX}) \leq \paren{1-\nu}^{k-K}\rho^2 ,
\end{align}
where
\begin{align}\label{conv_rate_G}
\nu = \frac{\para{1-\kappa}\Ppa{\lambda\min\pa{\normm{\avx}^2,1}-\vrho\max\pa{\normm{\avx}^2/3,1}}}{3+\vrho\max\pa{\normm{\avx}^2/3,1}} .
\end{align}
This holds with a probability at least $1-2e^{-\frac{m\pa{\sqrt{1+\vrho}-1}^2}{8}}-5e^{-\zeta n}-4/n^2-c/m$, where $C(\vrho),c$ and $\zeta$ are numerical positive constants.
\item \label{thm:Gauss-ii} \tcb{Suppose moreover that $\vrho$ obeys 
\begin{equation*}
\vrho\leq\eta_1^{-1}\Ppa{\frac{1-\lambda}{\sqrt{3\Ppa{6(1+(1-\lambda)^2/3)+1}}}\frac{1}{\max\Ppa{\normm{\avx},1}}} ,
\end{equation*}
where $\eta_1$ is the function defined in \eqref{radius}.}
When $m\geq C(\vrho)n\log(n)$, if Algorithm~\ref{alg:MDBT} is initialized with the spectral method in Algorithm~\ref{alg:algoSPIniit}, then with probability at least $1-2e^{-\frac{m\pa{\sqrt{1+\vrho}-1}^2}{8}}-5e^{-\zeta n}-4/n^2$ ($\zeta$ is a fixed numerical constant), \eqref{conv_lin_G} holds for all $k \geq K=0$.
\end{enumerate}
\end{theorem}

Before proving our result, the following remarks are in order.
\begin{remark} {\ }
\begin{itemize}
 \setlength{\itemindent}{0.2cm}
\item In the regime of claim~\ref{thm:Gauss-i}, when $x_0$ is chosen uniformly at random, Algorithm~\ref{alg:MDBT} provably converges to the true vector $\avx$ up to a sign change. In this case any initialization strategy becomes superfluous, though the number of measurements required then is slightly (polylogarithmically) higher than with spectral initialization.

\item In the regime of of claim~\ref{thm:Gauss-ii}, one has to use a spectral initialization to find a good initial guess, from which mirror descent converges locally linearly to $\avx$ up to global sign change.  

\item When the true vector norm is one, as assumed in many works, the convergence rate takes the simple form $\para{1-\frac{\para{1-\kappa}\pa{\lambda-\vrho}}{3+\vrho}} \leq \frac{2}{3} + O((1-\lambda)+\kappa+\vrho)$.

\item \tcb{The convergence rate $1-\nu$ as given in \eqref{conv_lin_G}-\eqref{conv_rate_G} can be slightly improved as we did in  \eqref{loclinear} (here we dropped the denominator in \eqref{loclinear}).} It is also important to point out that our convergence rate is independent from the dimension $n$ of the signal. This is in contrast with the Wirtinger flow \cite{Candes_WF_2015,candes_phase_2015}, whose convergence rate is $\paren{1-\frac{\mathrm{cst}}{n}}$ and thus dimension-dependent. Such dependence was removed for the truncated Wirtinger flow with Gaussian measurements \cite{chen_solving_2017}.
\end{itemize}

To close these remarks, we strongly believe that handling the geometry of the problem through the framework of mirror/Bregman gradient descent with a wisely chosen entropy/kernel $\psi$ is a key for this better behaviour in our case.
\end{remark}

\begin{proof}
$~$\\\vspace*{-0.25cm}
\begin{enumerate}[label=(\roman*)]
\setlength{\itemindent}{0.3cm}
\item 
Assume for this claim that $\calE_{\rm conv}$ holds true; we will show later that this is indeed the case \whp when the number of measurements is as large as prescribed. The proof then consists in combining Theorem~\ref{Maintheo} and the characterization of the structure of $\crit(f)$.
\begin{itemize}
\item Global convergence of the iterates: under event $\calE_{\rm inj}$ (see \eqref{eq:injectivity_G}), the sequence $\seq{\xk}$ generated by Algorithm~\ref{alg:MDBT} is bounded; see the discussion in Remark~\ref{rem:deterministic}.
Since $\calE_{\rm smad}$ holds, Theorem~\ref{Maintheo}\ref{point-i}-\ref{point-ii} ensure that the sequence $\seq{\xk}$ converges to $\xsol \in \crit(f)$ and the induced sequence $\seq{f(\xk)}$ converges to $f(\xsol)$.

\item Since $\calE_{\rm strictsad}$ holds also, we have by Theorem~\ref{Maintheo}\ref{point-iv} that for almost all initial points $x_0$, the sequence $\seq{\xk}$ converges to an element of $\crit(f)\bsl\strisad(f)=\overline{\calX}$. We assume \wlogp that $\xk \to \avx$ whence $\normm{\xk-\avx} \to 0$, and $f(\xk) \to \min(f)=0$. Therefore, for $\eta\leq\frac{\rho}{\sqrt{\max(\Theta(\rho),1)}}$, there exists $\exists K=K(\eta)$ such that, 
\begin{equation}
\forall k\geq K, \normm{\xk-\avx} < \eta \qandq f(\xk) \in ]0,\eta[, 
\end{equation} 
\ie for $k\geq K$, $\xk$ is in an $f$-attentive neighborhood of $\avx$. 

\item Thanks to $\calE_{\rm scvx}$, $f$ is $\sigma$-strongly convex \tcb{on $B(\avx,\rho)$ relative to $\psi$} with $\sigma$ and $\rho$ as given in that event. It then follows from Theorem~\ref{Maintheo}\ref{point-iii} that, $\forall k>K$ and $\gak \equiv \frac{\para{1-\kappa}}{3+\vrho\max\pa{\normm{\avx}^2/3,1}}$, we have 
\begin{align*}
D_{\psi}(\avx,\xkp)
&\leq \paren{1-\nu}D_{\psi}(\avx,\xk)\\
&\leq \paren{1-\nu}^{k-K}D_{\psi}(\avx,x_K).     
\end{align*}
Moreover by \eqref{striconv} and $1$-strong convexity of $\psi$, for all $k \geq K$
\begin{align*}
\dist^2(\xk,\overline{\calX})\leq\normm{\xk-\avx}^2\leq 2D_{\psi}(\avx,\xk)&\leq \paren{1-\nu}^{k-K}\Theta(\rho)\eta^2,\\
& \leq \paren{1-\nu}^{k-K} \rho^2.
\end{align*}
\end{itemize}

To conclude this part of the proof we need to compute the probability that the event $\calE_{\rm conv}$ occurs. 
We have, 
\begin{align*}
  \calE_{\rm conv}&=\calE_{\rm strictsad}\cap\calE_{\rm conH}\cap\calE_{\rm inj}\cap\calE_{\rm smad}\cap\calE_{\rm scvx},\\
  &= \calE_{\rm strictsad}\cap\calE_{\rm conH}\cap\calE_{\rm inj},
\end{align*}
since $\calE_{\rm smad}\subset\calE_{\rm conH}$ and $\calE_{\rm scvx}\subset\calE_{\rm conH}$  thanks to Lemma~\ref{pro:Lsmad_G} and Lemma~\ref{pro:Localconvex_G} respectively. Owing to Lemma~\ref{lem:conhess_G}, the event $\calE_{\rm conH}$ holds true with  a probability at least $1-5e^{-\zeta n}-\frac{4}{n^2}$, where $\zeta$ is a fixed numerical constant, with the proviso that $m \geq C(\vrho)n\log(n)$. 

On the other hand, Lemma~\ref{pro:injectivity_G} tells us that, when $m\geq\frac{16}{\vrho^2}n,$ the event $\calE_{\rm inj}$ is true with a probability at  least $1-2e^{-\frac{m\pa{\sqrt{1+\vrho}-1}^2}{8}}$. The study of the critical points of the objective $f$, see \cite[Theorem~2.2]{sun_geometric_2018}, shows that when $m\geq C(\vrho)n\log^3(n)$, the event $\calE_{\rm strictsad}$ holds true with a probability $1-\frac{c}{m}$ (where $c$ a fixed numerical constant). Using a union bound, $\calE_{\rm conv}$ occurs with the stated high probability provided that $m \geq C(\vrho)n\log(n)$ for a large enough numerical constant $C(\vrho)$.

\item The proof of this claim is similar to the last part of claim~\ref{thm:Gauss-i} except that now, we invoke Lemma~\ref{pro:spectralinit}\ref{pro:spectralinit-iii} to see that with probability at least at least $1-2e^{-\frac{m\pa{\sqrt{1+\vrho}-1}^2}{8}}-5e^{-\zeta n}-\frac{4}{n^2}$, the initial guess $\xo$ obtained by spectral initialization belongs to $B\Ppa{\overline{\calX},\frac{\rho}{\sqrt{\max(\Theta(\rho),1)}}}$. We can now follow the reasoning in the last item of the proof of statement~\ref{thm:Gauss-i} to conclude. 
\end{enumerate}
\end{proof}

\subsection{CDP measurements} 
Our main result for the CDP measurements model is the following. 
%

\begin{theorem}\label{thm:CDP}
Let $\vrho \in]0,1[$ and $\seq{\xk}$ be the sequence generated by Algorithm~\ref{alg:MDBT}.
\begin{enumerate}[label=(\roman*)]
\setlength{\itemindent}{0.3cm}
\item \label{thm:CDP_i} If the number of patterns $P$ satisfies $P\geq C(\vrho)\log(n)$, then with a probability at least $1-1/n^2$, for almost all initializers $x_0$ of Algorithm~\ref{alg:MDBT} used with constant step-size $\gamma_k \equiv \gamma = \frac{1-\kappa}{L}$, for any $\kappa \in ]0,1[$ and $L$ given by Lemma~\ref{Tsmad}, $\seq{\xk}$ converges to an element in $\crit(f)\bsl\strisad(f)$.

\item \label{thm:CDP_ii} Let $\delta\in]0,\min(\normm{\avx}^2,1)/2[$. There exists $\rho_{\delta} > 0$ such that if $\vrho$ is small enough (\ie it satisfies \eqref{eq:spectralinitvrhobnd_cdp}) and $P \geq C(\vrho)n\log^3(n)$, and if Algorithm~\ref{alg:MDBT} is initialized with the spectral method in Algorithm~\ref{alg:algoSPIniit}, then with probability at least $1-\frac{4P+1}{n^3}-\frac{1}{n^2}$
\begin{align}\label{conv_lin_cdp}
\dist^2(\xk,\overline{\calX}) \leq \prod_{i=0}^{k-1}(1-\nu_i) \rho_\delta^2, \quad \forall k \geq 0 ,
\end{align}
where 
\begin{align}\label{conv_rate_cdp}
\nu_i = \frac{\Ppa{1-\kappa}\Ppa{\min\pa{\normm{\avx}^2,1}-2\delta}}{(1+\delta)L_i} .
\end{align}

%
\end{enumerate}
\end{theorem}

Let us first discuss this result and compare it to the one for Gaussian measurements.
\begin{remark} {\ }
\begin{itemize}
 \setlength{\itemindent}{0.15cm}
\item As far as global recovery guarantees are concerned, Theorem~\ref{thm:CDP}\ref{thm:CDP_i} does not ensure exact recovery of $\pm\avx$. This is in contrast with the Gaussian model where this was established in Theorem~\ref{thm:Gauss}\ref{thm:Gauss-i}. As we pointed out earlier in the introduction section, one of the main difficulties is that several of our arguments in the Gaussian case rely on uniform bounds, for instance on the Hessian and gradient, that need to hold simultaneously for all vectors $x \in \bbR^n$ \whp. Unfortunately, the CDP model enjoys much much less randomness to exploit in the mathematical analysis making this very challenging. Nevertheless, numerical evidence in the next section suggests that global exact recovery (without spectral initialization) holds for the CDP model as well.
  
\item Theorem~\ref{thm:CDP}\ref{thm:CDP_ii} ensures local linear convergence to the true vectors $\pm\avx$ when our algorithm is initialized with the spectral method. The convergence rate is expressed in terms of the step-sizes $\gamma_i=\frac{1-\kappa}{L_i}$, where the $L_i$'s are expected to be much smaller than $L$ in Lemma~\ref{Tsmad}. It is tempting to use $2(1+\delta)^2$, the local relative smoothness constant in \eqref{eq:Lsmad_cdp}, as an upper-bound estimate of the $L_i$'s. But one has to keep in mind that this is valid only locally on $B(\pm \avx,\rho_{\delta})$, and thus one cannot use it when iterating from $\xk$ to $\xkp$. In our numerical experiments, we nevertheless observe that the linear convergence rate in \eqref{conv_rate_cdp} is well estimated by $\Ppa{1-\frac{\Ppa{1-\kappa}\Ppa{\min\pa{\normm{\avx}^2,1}-2\delta}}{2(1+\delta)^3}}$.  When $\normm{\avx} \leq 1$, this rate reads $\Ppa{1-\frac{\Ppa{1-\kappa}\Ppa{1-2\delta}}{2(1+\delta)^3}} \leq \frac{1}{2} + O(\kappa+\delta)$.

\end{itemize}
\end{remark}

\begin{proof}$~$ 
\begin{enumerate}[label=(\roman*)]
\setlength{\itemindent}{0.45cm}
\item 
Under the bound on $P$, we know from Lemma~\ref{pro:injectivity_cdp} that the measurement operator $A$ is injective  with probability at least $1-1/n^2$. On this event, the objective $f$ is coercive, and thus the sequence $\seq{\xk}$ generated by Algorithm~\ref{alg:MDBT} is bounded. Since $f$ is $L$-smooth \tcb{relative to} $\psi$ according to Lemma~\ref{Tsmad}, Theorem~\ref{Maintheo}\ref{point-i}-\ref{point-ii} ensure that the sequence $\seq{\xk}$ converges to $\xsol \in \crit(f)$ and the induced sequence $\seq{f(\xk)}$ converges to $f(\xsol)$. Then using Theorem~\ref{Maintheo}\ref{point-iv} we get the statement.
\item
By Lemma~\ref{pro:spectralinit_cdp}\ref{pro:spectralinit_cdp-iii}, we have that the spectral initialization guess $\xo$ belongs to \linebreak$B\Ppa{\overline{\calX},\frac{\rho_{\delta}}{\sqrt{\max(\Theta(\rho_{\delta}),1)}}}$ with probability larger than $1-\frac{4P+1}{n^3}-\frac{1}{n^2}$. Moreover, we know from Lemma~\ref{pro:Lsmad_cdp} that with probability at least $1-\frac{4P+1}{2n^3}$, $f$ is $\sigma$-strongly convex \tcb{on $B\pa{\overline{\calX},\rho_\delta}$ relative to} $\psi$ with $\sigma=\frac{\Ppa{\min\pa{\normm{\avx}^2,1}-2\delta}}{1+\delta}$. The rest  of the proof follows the same reasoning as in the last item of the proof of statement Theorem~\ref{thm:Gauss}\ref{thm:Gauss-i}. We omit the details. 
\end{enumerate}
\end{proof}
}

\section{Numerical experiments}\label{sec:numexp}
In this section, we discuss some numerical experiments to illustrate the efficiency of our phase recovery algorithm. We use the standard normal Gaussian and we consider the CDP model with a random ternary variable $d$, \ie taking values in $\{-1,0,1\}$ with probability $\{1/4,1/2,1/4\}$. In each instance, we measured the relative error between the reconstructed vector $x$ and the true signal one $\avx$ as
\begin{align}\label{relative_error}
\frac{\dist(x,\overline{\calX})}{\normm{\avx}}.
\end{align}
In the experiments, we set $\normm{\avx}=1$.

\subsection{Reconstruction of 1D signals}

\subsubsection{Gaussian measurements}
The goal is to recover a one-dimensional signal with $n = 128$ from Gaussians measurements. Figure~\ref{reconstruction_gauss}(a) shows the reconstruction result \tcb{from one random instance with} $m = 2 \times 128\times\log^3(128)$ without spectral initialization. Algorithm~\ref{alg:MDBT} was initialized with a vector drawn from the uniform distribution, and used with $600$ iterations and a constant step-size $\gamma=\frac{0.99}{3}$. \tcb{Given the oversampling rate, and as predicted by Theorem~\ref{thm:Gauss}\ref{thm:Gauss-i}, one can observe from Figure~\ref{reconstruction_gauss}(a) that we have exact recovery, and after $\sim 90$ iterations, the iterates enter a linear convergence regime. The "Theoretical error" corresponds to the linear convergence rate predicted by \eqref{conv_lin_G}-\eqref{conv_rate_G}, which is valid for $k$ large enough.}

Figure~\ref{reconstruction_gauss}(b) displays the results for the case where $m = 2 \times 128\times\log(128)$, and Algorithm~\ref{alg:MDBT} was applied with the same parameters as above except that the spectral initialization method was used to get the initial guess. As anticipated by Theorem~\ref{thm:Gauss}\ref{thm:Gauss-ii}, we again have exact recovery with a linear convergence behaviour starting from the initial guess. 

\begin{figure}
\centering
\subfloat[Reconstruction with random initialization]{
    \includegraphics[trim={2cm 0.4cm 0 0.4cm},clip,width=0.49\linewidth]{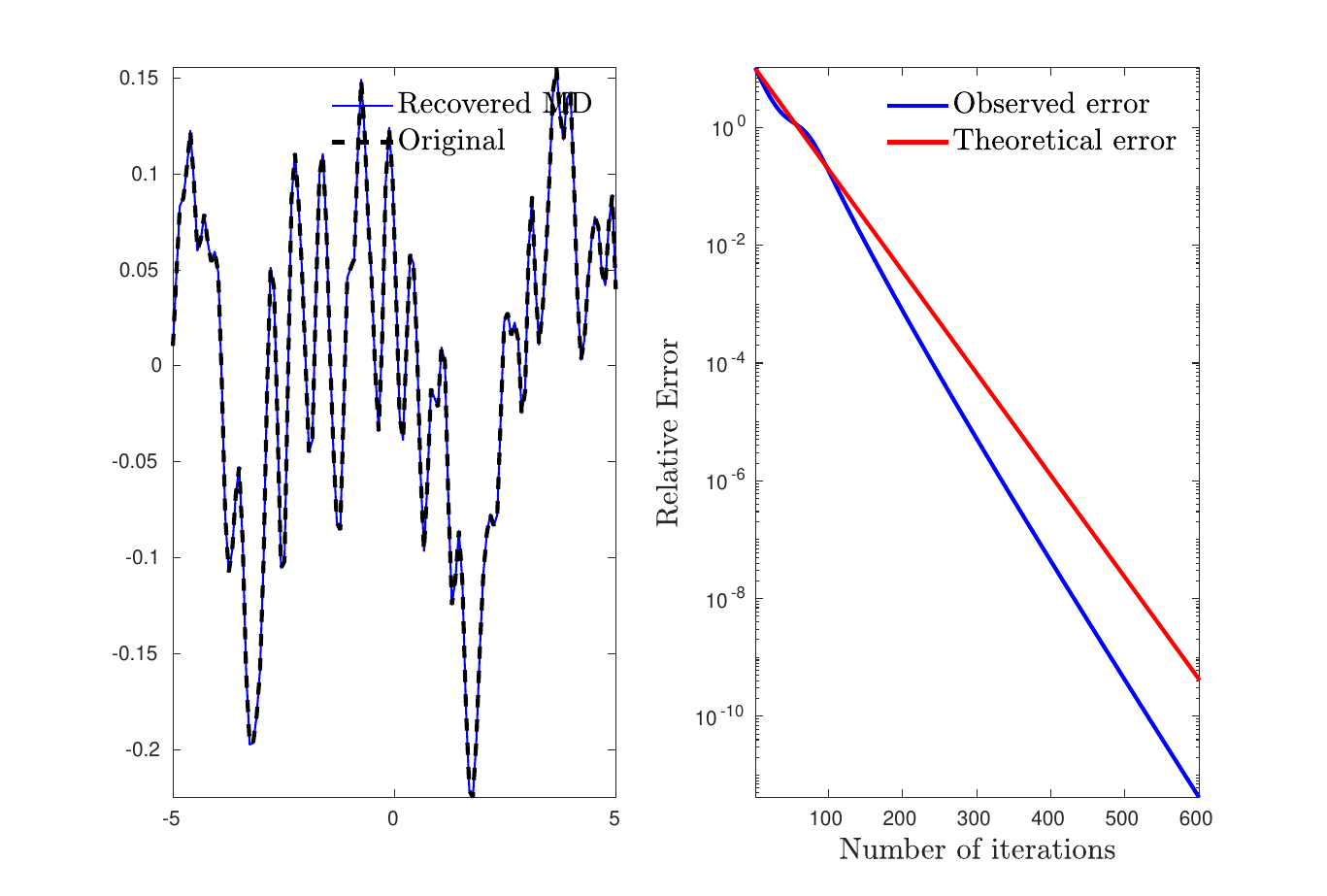}}
    \hspace{-0.5cm}
    \quad
    \subfloat[Reconstruction with spectral initialization]{
    \includegraphics[trim={2cm 0.4cm 0 0.4cm},clip,width=0.49\linewidth]{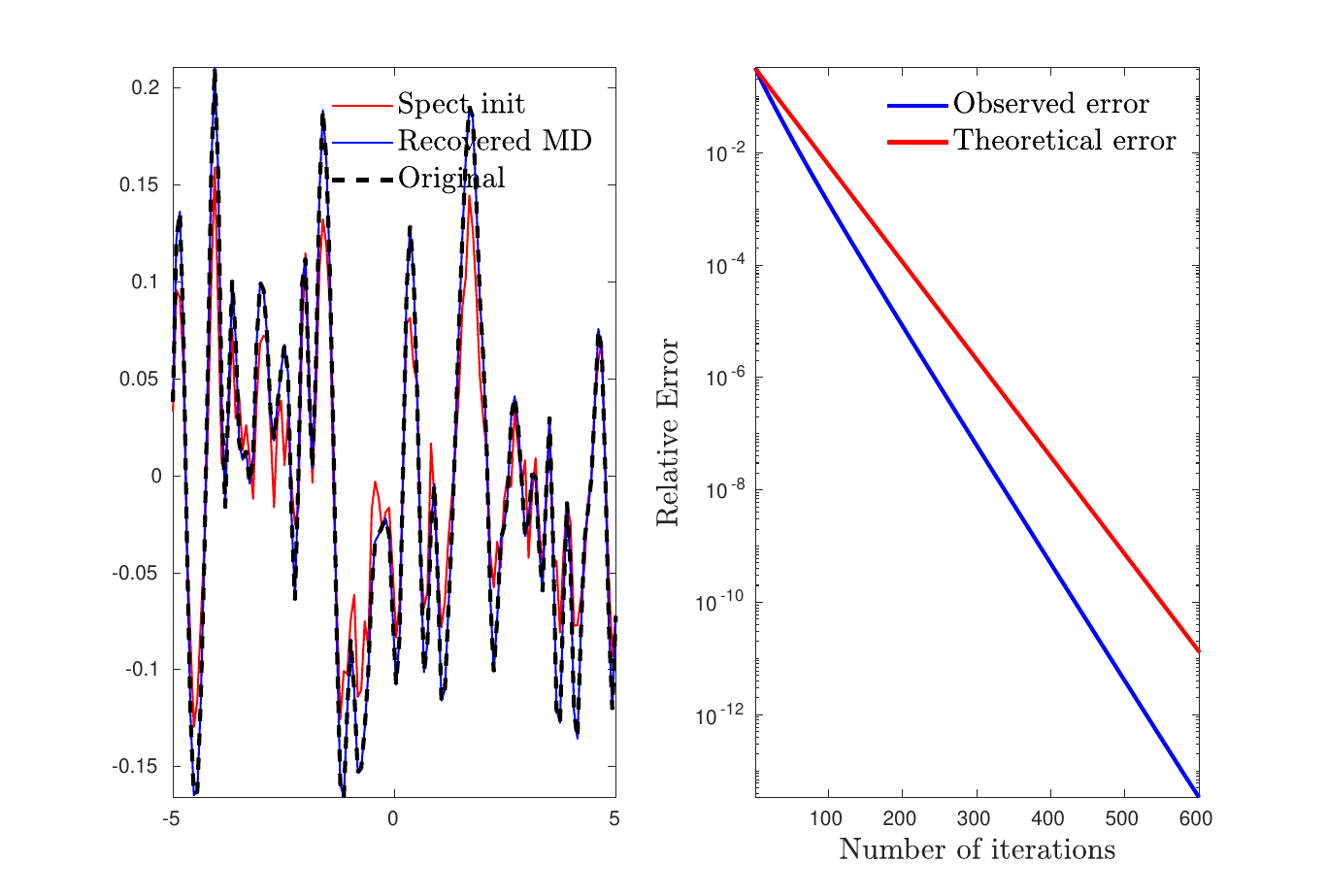}}
 \caption{Reconstruction of a 1D signal by mirror descent from Gaussian measurements.}
\label{reconstruction_gauss}
\end{figure}

\subsubsection{CDP measurements}
We carried out the same experiment with the CDP measurements where we took $P = 7 \times \log^3(128)$ \tcb{ternary random masks}, and set $\gamma=\frac{0.99}{2}$ in mirror descent. The results are shown in Figure~\ref{reconstruction_CDP}. The same conclusions drawn in the Gaussian case remain true for the CDP model. The results with spectral initialization depicted in Figure~\ref{reconstruction_CDP}(b) are in agreement with those of Theorem~\ref{thm:CDP}\ref{thm:CDP_ii}.
As for random uniform initialization, the results of Figure~\ref{reconstruction_CDP}(a) provide numerical evidence that our algorithm enjoys global exact recovery properties, though this is so far not justified by our theoretical analysis.   

\begin{figure}
\centering
\subfloat[Reconstruction with random initialization]{
    \includegraphics[trim={2cm 0.4cm 0 0.4cm},clip,width=0.5\linewidth]{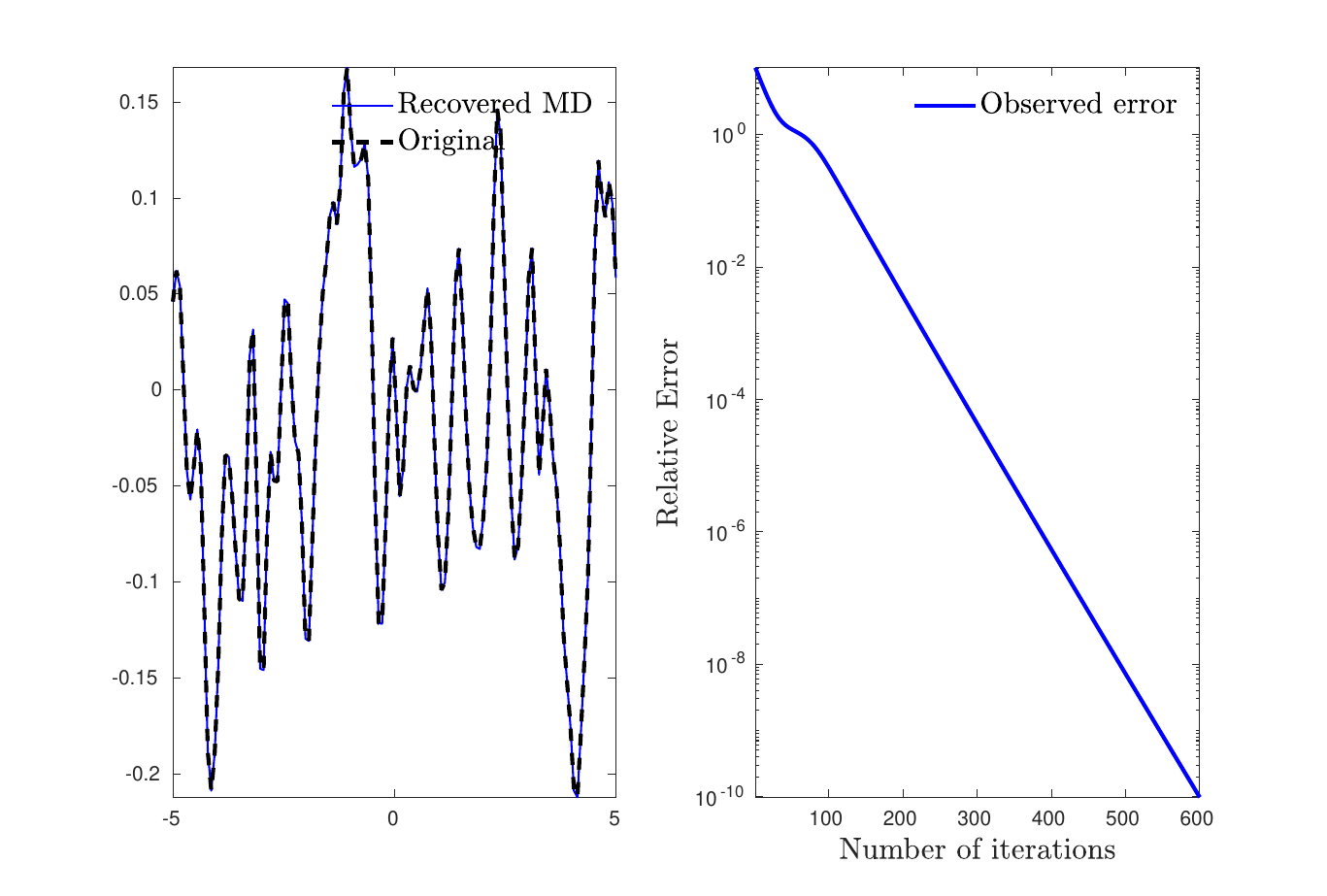}}\hspace{-1cm}\quad
    \subfloat[Reconstruction with spectral initialization]{
    \includegraphics[trim={2cm 0.4cm 0 0.4cm},clip,width=0.5\linewidth]{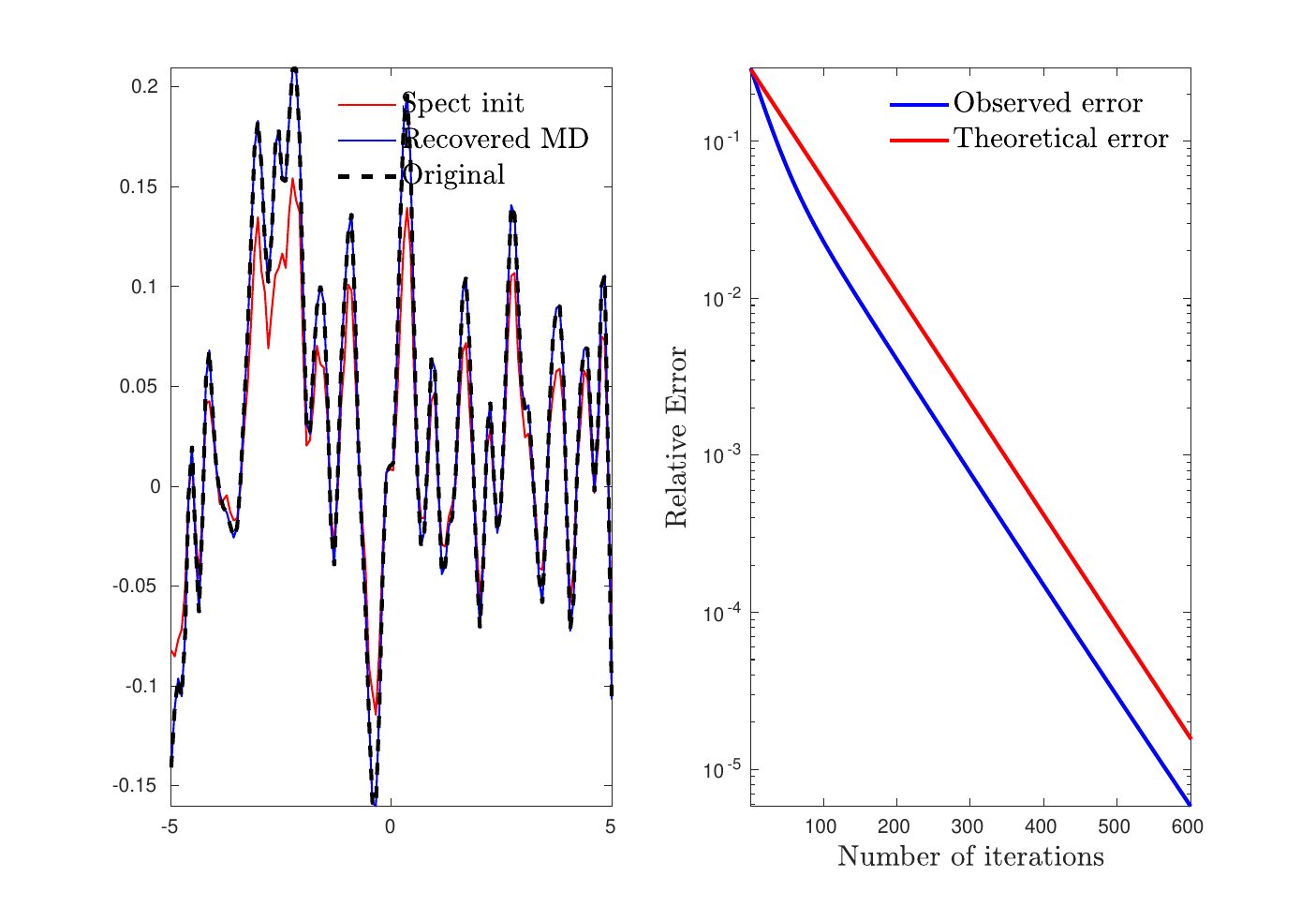}}
    \caption{Reconstruction of a 1D signal by mirror descent from CDP measurements.}
\label{reconstruction_CDP}
\end{figure}

\subsection{Recovery of the roughness of a 2D surface (light scattering)}
In this experiment, we simulated a rough surface as a $256 \times 256$ Gaussian random field. The goal to recover this surface profile from the magnitude of the measurements according to the CDP model with $P=100$ masks. The initial guess was drawn from the uniform distribution. The recovery results are displayed in Figure~\ref{reconstruction2D_CDP}. 

\begin{figure}
\centering
\subfloat[Original surface]{ 
	\centering
    \includegraphics[width=0.33\linewidth]{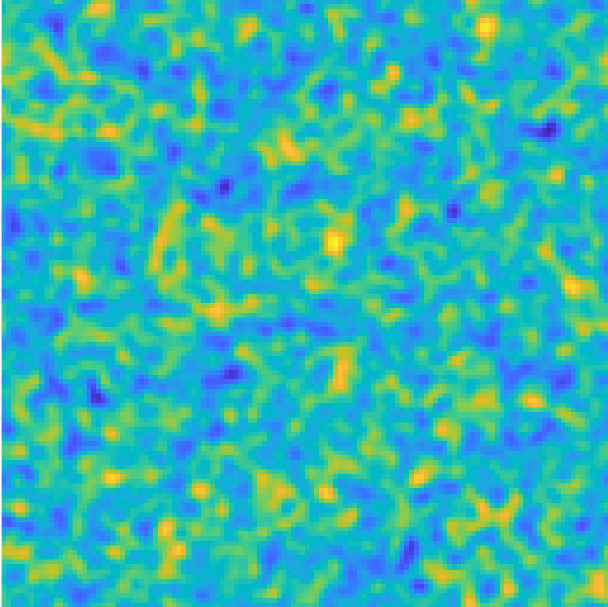}}
\subfloat[Recovered surface]{
    \centering
    \includegraphics[width=0.33\linewidth]{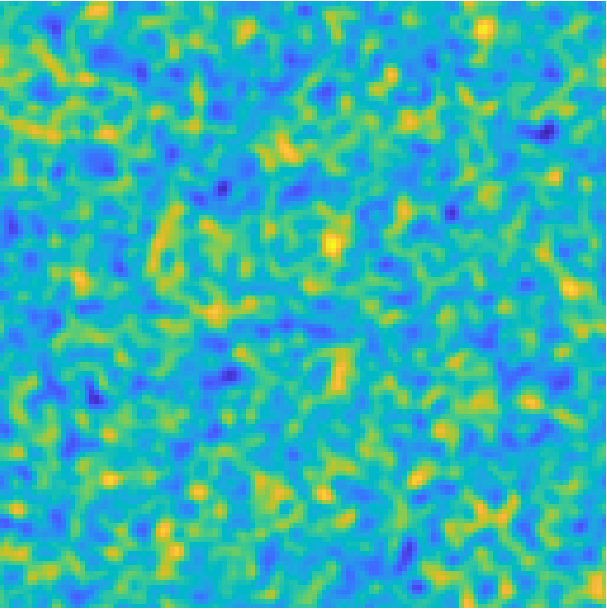}}
\subfloat[\tcb{Relative error}]{
    \centering
    \includegraphics[width=0.35\linewidth]{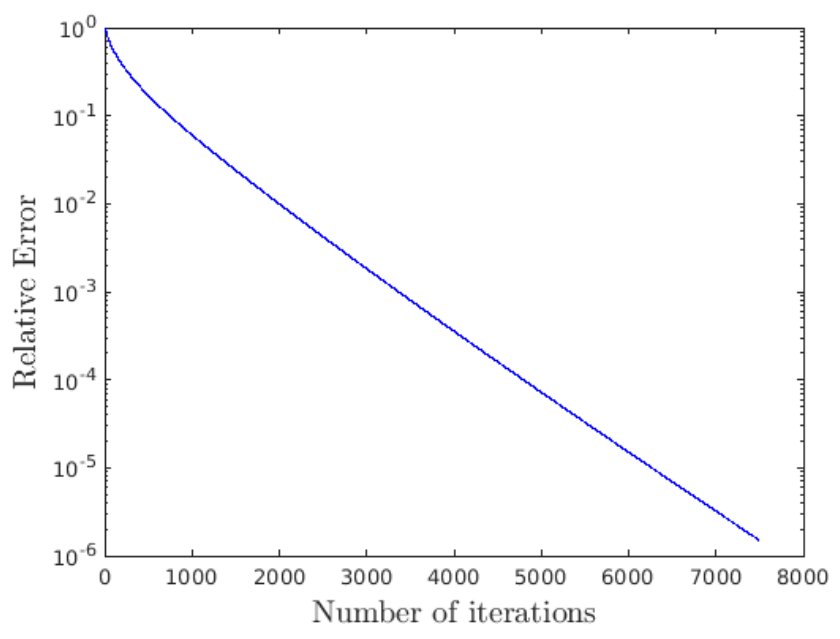}}
    \caption{Roughness surface profile reconstruction by solving the phase retrieval problem from the CDP measurement model using mirror descent with uniform random initialization.}
\label{reconstruction2D_CDP}
\end{figure}

\begin{figure}
\tcb{
\centering
	\subfloat[Gaussian measurements]{
    \centering
    \includegraphics[width=0.5\textwidth]{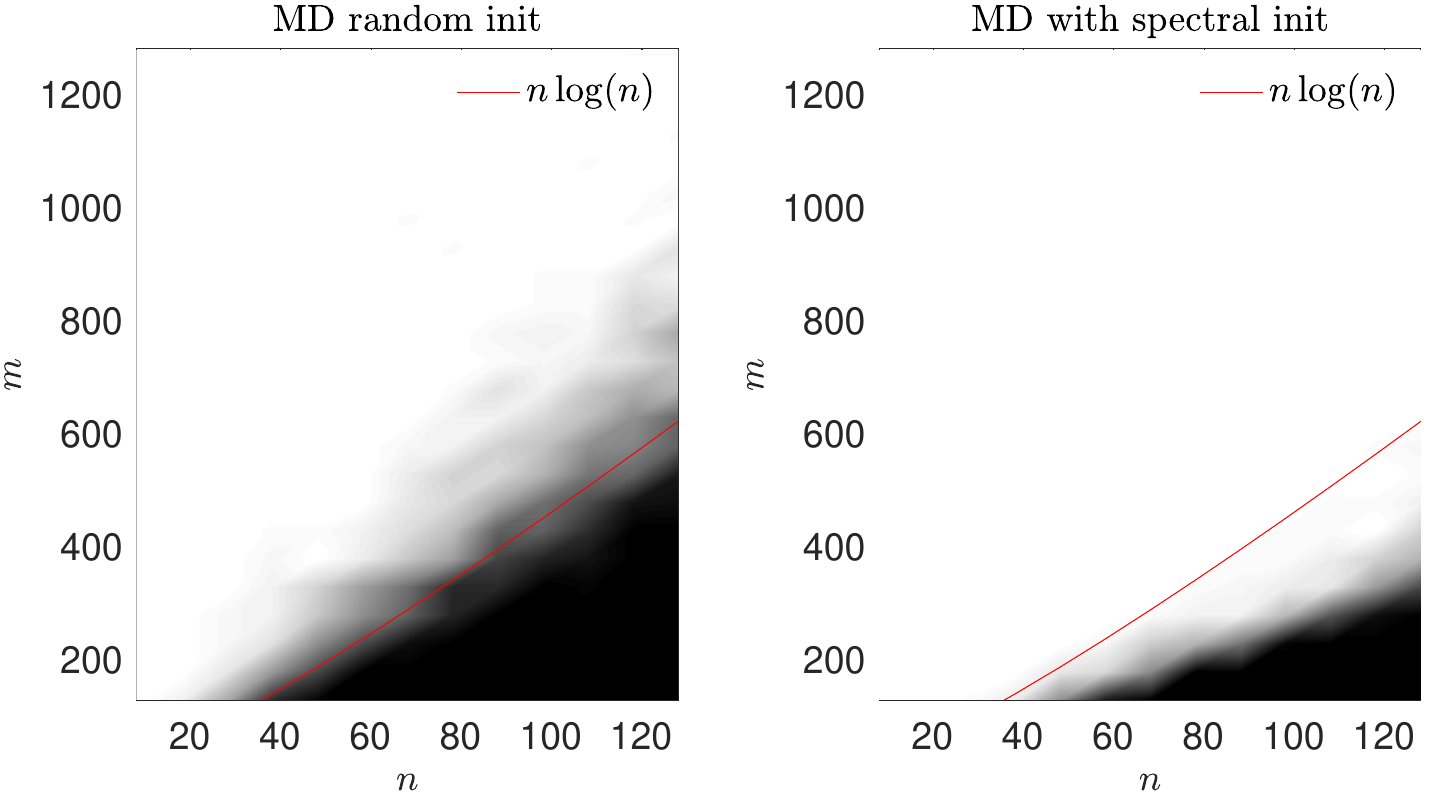}}
    \subfloat[CDP measurements]{
    \centering
    \includegraphics[width=0.5\textwidth]{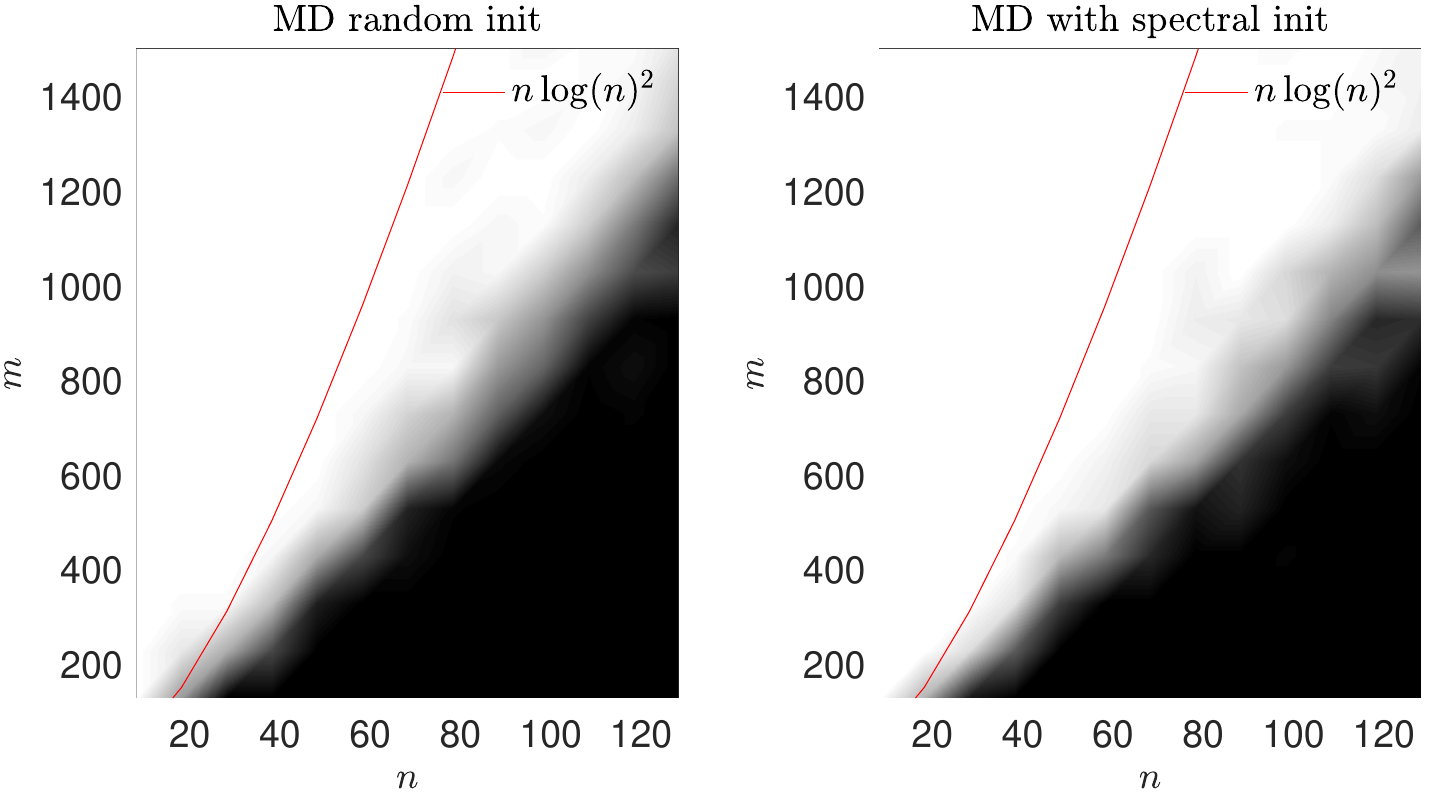}}}
    \caption{\tcb{Phase diagrams of mirror descent (MD) with spectral and uniform random initialization. (a) Gaussian measurements. (b) CDP measurements.}}
\label{Phase_transition_MD}
\end{figure}

\begin{figure}
\tcb{
\centering
	\subfloat[Gaussian measurements]{
    \centering
    \includegraphics[width=0.5\textwidth]{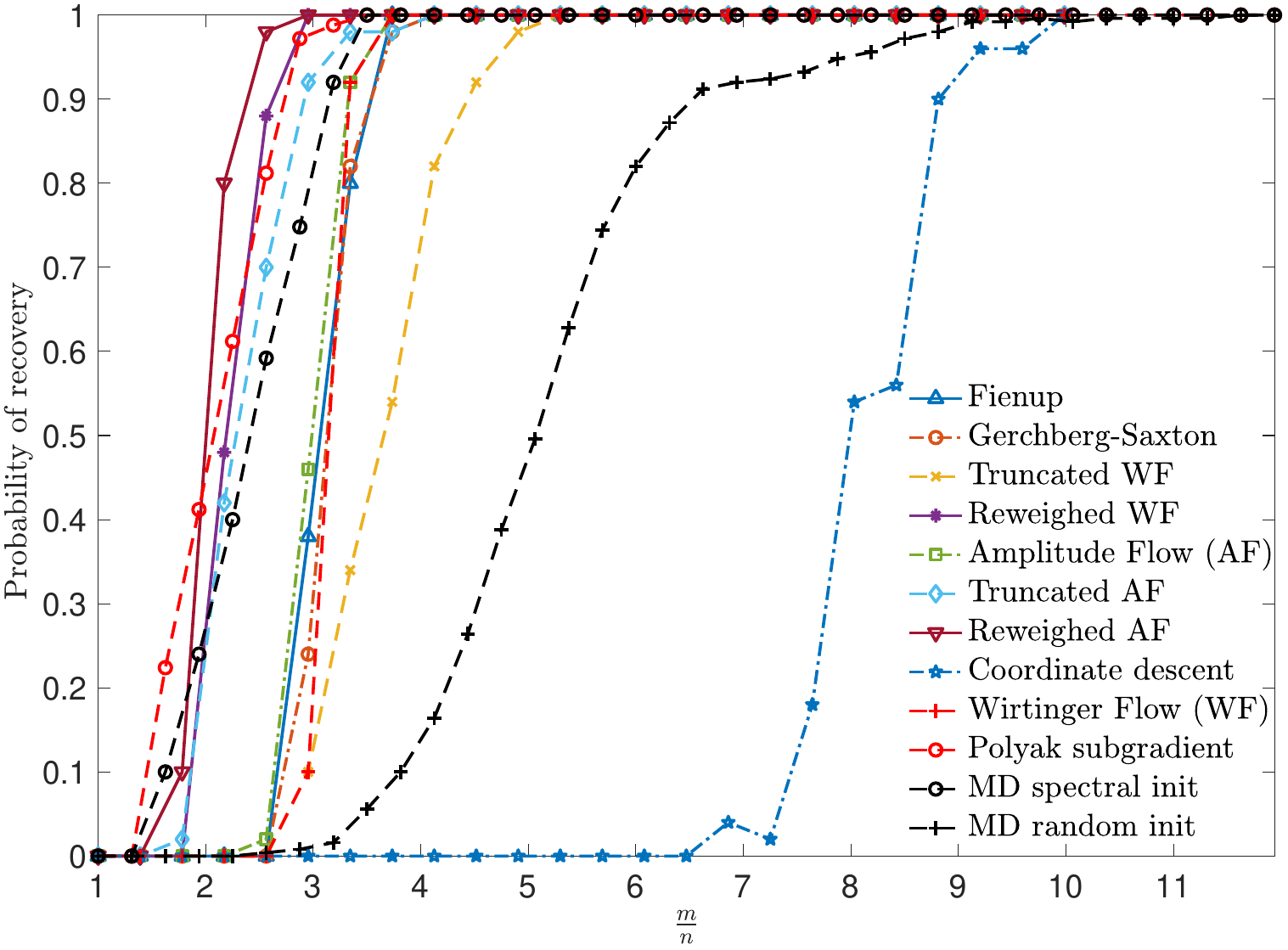}}
    \subfloat[CDP measurements]{
    \centering
    \includegraphics[width=0.5\textwidth]{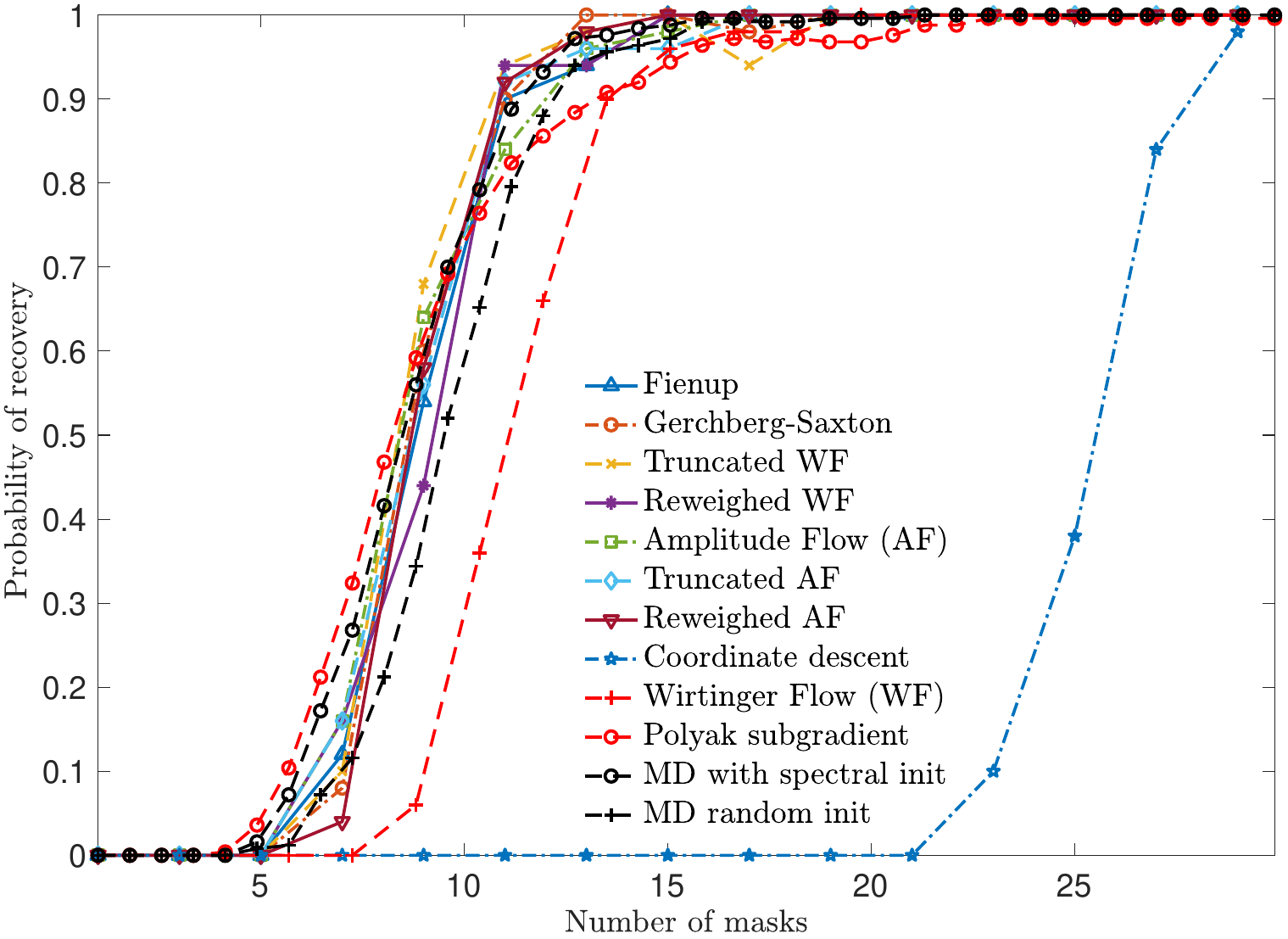}}}
\caption{\tcb{Comparison of mirror descent to other methods in the literature. Each plot shows the empirical probability of success based on 100 random trials for two different measurement models (Gaussian and CDP) and a varied number of measurements.}}
\label{Phase_transition_comparison}
\end{figure}

\subsection{Phase diagrams and comparison with other algorithms}
\tcb{
\paragraph{Phase diagrams}
We first report the results of an experiment designed to estimate the phase retrieval probability for mirror descent, as we vary $n$ and $m$. The results are depicted in Figure~\ref{Phase_transition_MD}. For each pair $(n,m)$, we generated $100$ random instances and solved them with mirror descent (denoted MD for short hereafter), both with spectral initialization and with random uniform initialization. Each diagram shows the empirical probability (among the $100$ random trials) that an algorithm successfully recovers the original vector up to a global sign change. We declared that a signal is recovered if the relative error \eqref{relative_error} is less than $10^{-5}$. The grayscale of each point in the diagrams reflects the empirical probability of success, from $0\%$ (black) to $100\%$ (white). The solid curve marks the prediction of the phase transition edge. One clearly sees a phase transition phenomenon which is in agreement with the predicted sample complexity bound shown as a solid line. For Gaussian measurements, MD with uniform random initialization has a transition to success occurring at a higher threshold compared to the version of MD with spectral initialization. This is in agreement with our theoretical findings. On the other hand, for CDP measurements, MD with uniform random initialization shows comparable performance to the version with spectral initialization especially as the oversampling (number of masks) increases, confirming numerically that spectral initialization does not seem to be mandatory for MD with CDP measurements.

\paragraph{Comparison with other algorithms}
We have also carried out a comprehensive comparative study of mirror descent (MD) to the methods included in the PhasePack library \cite{chandra2017phasepack}, which provides a common interface for testing phase retrieval methods on empirical datasets. We have used their implementations and included in the comparison MD and the Polyak subgradient method used in \cite{davis_nonsmooth_2020}. For fair comparison, and except MD with uniform initialization, we used spectral initialization for all algorithms. The results are displayed Figure~\ref{Phase_transition_comparison} where each plot shows the empirical probability of success of each algorithm based on 100 random trials for two different measurement models (Gaussian and CDP) and a varied number of measurements. We fixed $n=128$ in this experiment. References for all other algorithms as denoted in the legend in PhasePack can be found in \cite{chandra2017phasepack}.

For Gaussian measurements, MD with spectral initialization is in the group of best performing methods (Reweighted WF, Reweighted AF, Truncated AF, Polyak subgradient, MD) which exhibit comparable performance, though MD and Polyak subgradient are slightly better for low sampling rates (less than $2$), and Reweighted AF appears better for $m/n \in [2,3]$. This first group clearly outperforms the others especially when oversampling is less than $3$. This is followed by a second group (AF, Fineup, Gerchberg-Saxton and WF), then Truncated WF, MD with random initialization, and finally the Coordinate Descent method. As far CDP measurements are concerned, most algorithms perform similarly and MD with spectral initilization appears to be among the best ones. MD with uniform random initialization has a recovery performance rather close to those ones, and better than the Wirtinger flow even if the latter uses spectral initialization. 
}


\begin{appendices}\label{sec:appen}

\section{Proofs for the Deterministic Case}
{
Let us start this section by recalling our objective function \ie
\begin{equation}
\forall x\in\bbR^n,\quad f(x)=\qsom{\paren{|\adj{a_r}x|^2-y[r]}^2}=\qsom{\loss{|\adj{a_r}x|^2}{|\adj{a_r}\avx|^2}}, 
\end{equation}

The following expressions give the gradients and Hessians of $f$ and $\psi$ that will be used throughout. For all $\forall x\in\bbR^n$, we have

\begin{align}
\nabla f(x)&= \som{\paren{|\adj{a_r}x|^2-|\adj{a_r}\avx|^2}a_r\adj{a_r}x}, &\nabla^2f(x)&= \som{\paren{3|\adj{a_r}x|^2-|\adj{a_r}\avx|^2}a_r\adj{a_r}},\label{hessg}\\
\nabla\psi(x)&=\paren{\normm{x}^2+1}x, &\nabla^2\psi(x)&=\paren{\normm{x}^2+1}\Id+2x\transp{x}.\label{hessent}
\end{align}

Let start with the following useful lemma to compare the Bregman divergences of smooth functions.
\begin{lemma}\label{Bregcomp}
Let $g,\phi\in C^2(\bbR^n)$. If $\forall u\in\bbR^n$, $\nabla^2g(u) \lon \nabla^2\phi(u)$ for all $u$ in the segment $[x,z]$, then,
\begin{equation}
D_g(x,z)\leq D_{\phi}(x,z) .
 \end{equation}
\end{lemma}
\begin{proof}
The result comes from the Taylor-MacLaurin expansion. Indeed we have $\forall x,z\in \bbR^n$ 
\begin{align*}
D_g(x,z)
&=g(x)-g(z)-\pscal{\nabla g(z),x-z} \\
&=\int_0^1(1-\tau)\pscal{x-z,\nabla^2g(z+\tau(x-z))(x-z)}d\tau,
\end{align*}
and thus
\begin{multline*}
D_{\phi}(x,z)-D_{g}(x,z) = \\ \int_0^1(1-\tau)\pscal{x-z,\paren{\nabla^2\phi(z+\tau(x-z))-\nabla^2g(z+\tau(x-z))}(x-z)}d\tau .
\end{multline*}
The positive semidefiniteness assumption implies the claim.
\end{proof}
}

{
\subsection{Proof of Lemma~\ref{Tsmad}}\label{PrTsmad}
\begin{proof}
Our proof is different from that of \cite[Lemma~5.1]{bolte_first_2017} and gives a better estimate of $L$. Since $y$ has positive entries, we have for all $x,u \in \bbR^n$,
\begin{align*}
\pscal{u,\nabla^2f(x)u} 
&= \som{\paren{3|\adj{a_r}x|^2-y[r]}|\adj{a_r}u|^2}\\
&\leq \som{3|\adj{a_r}x|^2|\adj{a_r}u|^2} \\
&\leq \normm{x}^2\normm{u}^2\som{3\normm{a_r}^4} .
\end{align*}
On the other hand,
\begin{align*}
\pscal{u,\nabla^2\psi(x)u} 
&= \paren{\normm{x}^2+1}\normm{u}^2 + 2 |\pscal{x,u}|^2 \\
&\geq \normm{x}^2\normm{u}^2
\end{align*}
Thus for any $L \geq \som{3\normm{a_r}^4}$, we have for all $x \in \bbR^n$
\begin{equation}\label{hessmat}
\nabla^2f(x) \lon L\nabla^2\psi(x). 
\end{equation}
We conclude by invoking Lemma~\ref{Bregcomp} with $g=f$ and $\phi=L\psi$, and Proposition~\ref{pp:bregman}\ref{pp:bregman2}. 
\end{proof}

\tcb{
The following lemma states a key inequality that will be the starting point of our proof. It has appeared in different forms in the literature; see \cite[Lemma~4.1 and Remark~4.1]{bolte_first_2017} or \cite[Lemma~4.1]{Teboulle18}. We hereafter include a self-contained proof that accounts for backtracking.
\begin{lemma}\label{Alllinone}
Let $\seq{\xk}$ be a sequence  generated by Algorithm~\ref{alg:MDBT}. Then $\forall x \in \bbR^n$
\begin{align}\label{Allinone}
D_{\psi}\paren{x,\xkp} + \gamma_k\paren{f(\xkp)-f(x)}\leq D_{\psi}\paren{x,\xk}-\kappa D_{\psi}\paren{\xkp,\xk}-\gamma_k D_f\paren{x,\xk} .
\end{align}
\end{lemma}
}
%
%
\begin{proof}
From the update of $\xkp$, we have $\nabla\psi(\xk)-\nabla\psi(\xkp)=\gamma_k\nabla f(\xk)$, and multiplying both sides by  $\xkp-x$, we get 
\begin{align}\label{psdir}
\pscal{\nabla\psi(\xk)-\nabla\psi(\xkp),\xkp-x}=\gamma_k\pscal{\nabla f(\xk),\xkp-x}. 
\end{align}
Using the three-point identity \eqref{3poin}, we have
\begin{align}\label{MD3poin}
D_{\psi}\paren{x,\xk}-D_{\psi}\paren{x,\xkp}-D_{\psi}\paren{\xkp,\xk}=\gamma_k\pscal{\nabla f(\xk),\xkp-x}, 
\end{align}
By the backtracking test, we have that $f$ verifies the $L_k-$relative smoothness inequality \eqref{eq:smoothadaptable} \wrt $\psi$ at $(\xkp,\xk)$, with constant $L_k \leq L$, that is
\begin{align*}
f(\xkp)-f(\xk) 
&\leq \pscal{\nabla f(\xk),\xkp-\xk} + L_k D_{\psi}(\xkp,\xk) \\
& = \pscal{\nabla f(\xk),\xkp-x}+\pscal{\nabla f(\xk),x-\xk}+ L_k D_{\psi}(\xkp,\xk),\label{Tsmadreplace}\numberthis
\end{align*}
Plugging \eqref{MD3poin} into \eqref{Tsmadreplace}, we arrive at
\begin{align*}
&\gamma_k\paren{f(\xkp)-f(\xk)} \\
&\leq D_{\psi}\paren{x,\xk}-D_{\psi}\paren{x,\xkp}-D_{\psi}\paren{\xkp,\xk}+\gamma_k\pscal{\nabla f(\xk),x-\xk}+ \gamma_k L_k D_{\psi}(\xkp,\xk)\\
&\leq D_{\psi}\paren{x,\xk}-D_{\psi}\paren{x,\xkp}-\paren{1-\gamma_k L_k}D_{\psi}\paren{\xkp,\xk}+\gamma_k\pscal{\nabla f(\xk),x-\xk} \\
&\leq D_{\psi}\paren{x,\xk}-D_{\psi}\paren{x,\xkp}-\kappa D_{\psi}\paren{\xkp,\xk}+\gamma_k\pscal{\nabla f(\xk),x-\xk}.
\end{align*}
Therefore
\begin{align*}
\gamma_k\paren{f(\xkp)-f(x)} 
&\leq D_{\psi}\paren{x,\xk}-D_{\psi}\paren{x,\xkp}-\kappa D_{\psi}\paren{\xkp,\xk}\\
&\quad +\gamma_k\paren{f(\xk)-f(x)+\pscal{\nabla f(\xk),x-\xk}}\\ 
&= D_{\psi}\paren{x,\xk}-D_{\psi}\paren{x,\xkp}-\kappa D_{\psi}\paren{\xkp,\xk}-\gamma_k D_f(x,\xk) .
\end{align*}
\end{proof}
}

{
\subsection{Proof of Theorem~\ref{Maintheo}}\label{PrMaintheo}
\begin{proof} $~$
\paragraph*{\ref{point-i}-\ref{point-ii}} The objective function $f$ in \eqref{formulepro} is a real polynomial, hence obviously semi-algebraic. It then follows that $f$ satisfies the Kurdyka-\L{}ojasiewicz (KL) property \cite{loj1,loj2}. Combining this with Lemma~\ref{Alllinone}, which ensures that the sequence $(\xk)_{k \in \bbN}$ is a gradient-like descent sequence, and $1$-strong convexity of the entropy $\psi$, the proof of \ref{point-i}-\ref{point-ii} are similar to those of \cite[Proposition~4.1,Theorem~4.1]{bolte_first_2017} with slight modifications to handle backtracking. 

\paragraph*{\ref{point-iii}-\ref{point-iii-a}} The proof of this claim follows the same steps as the proof of \cite[Theorem~2.12]{attouch_convergence_2013} using again that $f$ is a continuous function which satisfies the KL property, that $\min f = 0$ and that $(\xk)_{k \in \bbN}$ is a gradient-like descent sequence thanks to Lemma~\ref{Alllinone}.

\paragraph*{\ref{point-iii}-\ref{point-iii-b}}
We verify by induction that $\xk \in B(\xsol,\rho), \forall k \in \bbN$. Observe first that $\xo \in B(\xsol,r) \subset B(\xsol,\rho)$ since $r \leq \frac{\rho}{\max\pa{\sqrt{\Theta(\rho)},1}}\leq\rho$. Suppose now that for $k \geq 0$, $x_i \in B(\xsol,\rho)$ for all $i \leq k$. From Lemma~\ref{Alllinone} applied at $x=\xsol$, and the optimality of $\xsol$, we have
\begin{align}
D_{\psi}\paren{\xsol,\xkp}&\leq D_{\psi}\paren{\xsol,\xkp}-\paren{1-\gamma_{k} L} D_{\psi}\paren{\xkp,\xk}-\gamma_{k} D_f(\xsol,\xk) \nonumber\\
&\leq D_{\psi}\paren{\xsol,\xk}-\gamma_{k} D_f(\xsol,\xk) \nonumber\\
&\leq \paren{1-\gamma_{k}\sigma}D_{\psi}\paren{\xsol,\xk} \label{eq:Dpsiloclinear}\\
&\leq \prod_{i=0}^{k}(1-\gamma_i\sigma) D_\psi(\xsol,\xo)
\leq D_{\psi}\paren{\xsol,\xo} , \nonumber
\end{align}
where we used the positivity of $D_{\psi}$ and the \tcb{relative strong convexity on $B(\xsol,\rho)$}. Now invoking Proposition~\ref{pp:bregman}\ref{pp:bregman4}, we have 
\begin{align*}
\normm{\xkp-\xsol}^2\leq 2 D_{\psi}\paren{\xsol,\xk}
&\leq 2\prod_{i=0}^{k}(1-\gamma_i\sigma) D_\psi(\xsol,\xo) \\
&\leq \Theta(\rho) \normm{\xo-\xsol}^2
\leq \frac{\Theta(\rho)}{\max\pa{\Theta(\rho),1}} \rho^2 \leq \rho^2 ,
\end{align*}
which entails that $x_i \in B(\xsol,\rho)$ for all $i \leq k+1$ as desired. \\
\tcb{
To show \eqref{loclinear}, we use again Lemma~\ref{Alllinone}, relative strong convexity on $B(\xsol,\rho)$, and \eqref{eq:Dpsiloclinear} to get 
\begin{align}\label{eq:Dpsiloclinearxkp}
D_{\psi}\paren{\xsol,\xkp} + \gamma_k\sigma D_{\psi}\paren{\xkp,\xsol}
\leq D_{\psi}\paren{\xsol,\xkp} + \gamma_k\paren{f(\xkp)-\fsol} 
\leq \paren{1-\gamma_{k}\sigma}D_{\psi}\paren{\xsol,\xk} .
\end{align}
Now Proposition~\ref{pp:bregman}\ref{pp:bregman4} and $1$-strong convexity of $\psi$ tell us that 
\begin{equation}\label{eq:sympsi}
D_{\psi}\paren{\xsol,\xkp} \leq \Theta(\rho)D_{\psi}\paren{\xkp,\xsol} .
\end{equation}
Combining \eqref{eq:Dpsiloclinearxkp}, \eqref{eq:sympsi}, $1$-strong convexity of $\psi$ and that $2D_{\psi}\paren{\xsol,\xo} \leq \rho^2$, we get the claim.
}
\paragraph*{\ref{point-iv}} We need the following lemma which is an extension of \cite[Proposition~10]{Lee19} to the more general \tcb{$L-$smooth} case. 
\begin{lemma}\label{lem:strictsadeig}
 Let $F$ be defined as in \eqref{MDalgo} then,
 \begin{enumerate}[label=(\alph*)]
  \setlength{\itemindent}{0.35cm}
  \item $\forall x\in\bbR^n, \det{\deriv F(x)}\neq0,$
  \item 
  $
   \strisad(f) \subset U_F \eqdef \enscond{x\in\bbR^n}{F(x)=x, \max_{i}\left|\lambda_i(\deriv F(x))\right|>1} .
  $
 \end{enumerate}
\end{lemma}
\begin{proof}[Proof of Lemma~\ref{lem:strictsadeig}] 
Recall that $F(x)=(\nabla\psi)^{-1}\paren{\nabla\psi(x)-\gamma\nabla f(x)}$.
Denote $G(x) \eqdef \nabla\psi(x)-\gamma\nabla f(x)$ so that $F(x)=(\nabla\psi)^{-1} \circ G(x)$.
\begin{enumerate}[label=(\alph*)]
  \setlength{\itemindent}{0.3cm}
\item Since $\psi$ is $C^2$ function, and thus $\nabla\psi$ is $C^1$, and as $\psi$ is strongly convex, the inverse function theorem ensures that $(\nabla\psi)^{-1}$ is a local diffeomorphism \footnote{Recall that we have already argued that $\psi$ is a Legendre function and thus $\nabla \psi$ is a bijection from $\bbR^n$ to $\bbR^n$ with inverse $(\nabla\psi)^{-1}=\nabla\psi^*$; see \cite[Theorem~26.5]{rockafellar_convex_1970}}. Therefore to have $\det{\deriv F(x)}\neq0$, it suffices to show that $G$ is a local diffeomorphism \ie $\forall x\in\bbR^n,\deriv {G(x)}$ is an invertible linear transformation. We have $\deriv{G(x)}= \nabla^2\psi(x)-\gamma\nabla^2f(x)$, and 
the $L-$relative smoothness property of $f$ \wrt $\psi$ (see \eqref{hessmat} in the proof of Lemma~\ref{Tsmad}) implies that 
\begin{align*}
\deriv{G(x)} = \nabla^2\psi(x)-\gamma\nabla^2f(x) \slon  (1-\gamma L) \nabla^2\psi(x) = \kappa \nabla^2\psi(x) \slon \kappa \Id \succ 0 .
\end{align*}
where we used $1$-strong convexity of $\psi$ and that $\gamma L = 1-\kappa \in ]0,1[$.

\item For $\xpa\in\strisad(f)$, we have $F(\xpa)=\xpa$ since $\strisad(f)\subset\crit(f)$. It remains to show that  $\det{\deriv F(\xpa)}$ has an eigenvalue of magnitude greater than one.
We have, 
\begin{align*}
\deriv F(\xpa)\stackrel{\mathrm{(Chain\text{ }rule)}}=&\nabla^2\psi^{-1}(G(\xpa))\deriv{G(\xpa)},\\
 =&\nabla^2\psi^{-1}(\xpa)\paren{\nabla^2\psi(\xpa)-\gamma\nabla^2f(\xpa)},\\
 =&\Id-\gamma\nabla^2\psi(\xpa)^{-1}\nabla^2f(\xpa).
\end{align*}
Denote for short $H_{\psi}=\nabla^2\psi(\xpa)$. We then have
\begin{align*}
H_{\psi}^{1/2}\deriv F(\xpa)H_{\psi}^{-1/2}= \Id-\gamma H_{\psi}^{-1/2}\nabla^2f(\xpa)H_{\psi}^{-1/2}.
\end{align*}
$H_{\psi}^{1/2}\deriv F(\xpa)H_{\psi}^{-1/2}$ is symmetric. Let $v'=H_{\psi}^{1/2}v$ with $v$ a unit-norm eigenvector associated to a strictly negative eigenvalue of $\nabla^2f(\xpa)$. By the Courant-Fisher min-max theorem, we have
\begin{align*}
\lambda_{\min}(H_{\psi}^{-1/2}\nabla^2f(\xpa)H_{\psi}^{-1/2}) 
&\leq \pscal{v',H_{\psi}^{-1/2}\nabla^2f(\xpa)H_{\psi}^{-1/2}v'} \\
&= \pscal{v,\nabla^2f(\xpa)v} < 0 .
\end{align*}
In turn, $1-\gamma\lambda_{\min}(H_{\psi}^{-1/2}\nabla^2f(\xpa)H_{\psi}^{-1/2}) > 1$ is an eigenvalue of $H_{\psi}^{1/2}\deriv F(\xpa)H_{\psi}^{-1/2}$. Since, 
$H_{\psi}^{1/2}\deriv F(\xpa)H_{\psi}^{-1/2}$ is similar to $\deriv F(\xpa)$, we conclude.
 \end{enumerate}
\end{proof}
To show \ref{point-iv}, we combine claim \ref{point-ii}, Lemma~\ref{lem:strictsadeig} and  the centre stable manifold theorem (see \cite[Corollary~1]{Lee19}) which allows to conclude that $\enscond{\xo\in\bbR^n}{\lim\limits_{k\rightarrow\infty} F^k(\xo)\in\strisad(f)}$ has measure zero. 
\end{proof}
}

{
\section{Proofs for Random Measurements}
\subsection{Gaussian measurements}
In this section, we assume that the sensing vectors $(a_r)_{r\in \tcb{\bbrac{m}}}$ follow the \iid standard Gaussian model.
\subsubsection{Expectation and deviation of the Hessian}
The next lemma gives the expression of the expectation of $\nabla^2 f(x)$. 
\begin{lemma}\label{lem:expe_lem_G}\textbf{(Expectation of the Hessian)}
Under the Gaussian model, we have
\begin{align}\label{eq:expe_lem_G}
\esp{\nabla^2 f(x)}=3\para{2x\transp{x}+ \normm{x}^2\Id}-2\avx\transp{\avx}-\normm{\avx}^2\Id. 
\end{align}
\end{lemma}
\begin{proof} 
In view of \eqref{hessg}, it is sufficient to compute
\begin{align*}
\esp{\som{|\transp{a_r}x|^2}a_r\transp{a_r}} .
\end{align*}
Computing this expectation is standard using independence and a simple moment calculation, which gives 
\begin{align}
\esp{\som{|\transp{a_r}x|^2}a_r\transp{a_r}}=2x\transp{x}+\normm{x}^2\Id.
\end{align}
\end{proof}

We now turn our attention to the concentration of the Hessian of $f$ around its mean. We start with following key lemma.
\begin{lemma}\label{lem:conhess_G_int}
Fix $\vrho\in]0,1[$. If the number of samples obeys $m \geq C(\vrho)n\log n$, for some sufficiently large $C(\vrho) > 0$, then 
\begin{align*}
\normm{\som{|\transp{a_r}x|^2a_r\transp{a_r}}-\paren{2x\transp{x}+\normm{x}^2\Id}}\leq \frac{\vrho}{3} \normm{x}^2.
\end{align*}
holds simultaneously for all $x\in\bbR^n$ with a probability at least $1-5e^{-\zeta n}-\frac{4}{n^2}$,  where $\zeta$ is a fixed numerical constant. 
\end{lemma}
\begin{proof} 
We follow a similar strategy to that of \cite[Section~A.4]{Candes_WF_2015}. By a homogeneity argument and isotropy of the Gaussian distribution, it is sufficient to establish the claim for $x=e_1$, \ie that
\begin{align}\label{eq:bndsspectmtxgaussian}
\normm{\som{|a_r[1]|^2a_r\transp{a_r}}-\paren{2e_1\transp{e_1}+\Id}}\leq \frac{\vrho}{3} .
\end{align}
Since the matrix in \eqref{eq:bndsspectmtxgaussian} is symmetric, its spectral norm can be computed via the associated quadratic form, and \eqref{eq:bndsspectmtxgaussian} amounts to showing that
\begin{align*}
V(v) \eqdef \left|\som{|a_r[1]|^2|\transp{a}_r v|^2}-\paren{1+2v[1]^2}\right| \leq \frac{\vrho}{3}
\end{align*}
for all $v\in\bbS^{n-1}$. The rest of the proof shows this claim.

Let $\wtilde{a}_r=\paren{a_r[2],\ldots,a_r[n]}$ and $\wtilde{v}=\paren{v[2],\ldots,v[n]}.$ We rewrite 
\[
|\transp{a}_rv|^2=\paren{a_r[1]v[1]+\transp{\wtilde{a}}_r\wtilde{v}}^2=\paren{a_r[1]v[1]}^2+\paren{\wtilde{a}_r^\top\wtilde{v}}^2+2a_r[1]v[1]\wtilde{a}_r^\top\wtilde{v} .
\]
We plug this decomposition into $V(v)$ to get 
\begin{align*}
V(v)
&=\left|\som{a_r[1]^4v[1]^2}+\som{a_r[1]^2(\wtilde{a}_r^\top\wtilde{v})^2}+2\som{|a_r[1]|^3v[1]\wtilde{a}_r^\top\wtilde{v}}-\paren{\normm{\tilde{v}}^2+3v[1]^2}\right|,\\
&\leq \left|\som{a_r[1]^4-3}\right|v[1]^2 +\left|\som{a_r[1]^2-1}\right|\normm{\wtilde{v}}^2+2\left|\som{|a_r[1]|^3v[1]\wtilde{a}_r^\top\wtilde{v}}\right| \\
&+\left|\som{a_r[1]^2\paren{\wtilde{a}_r^\top\wtilde{v}-\normm{\tilde{v}}^2}}\right|. 
\end{align*}
If $X \sim \mathcal{N}(0,1)$ we have $\esp{X^{2p}}=\frac{(2p)!}{2^pp!}$ for $p \in \N$, and in particular $\esp{X^2}=1$ and $\esp{X^4}=3$. By the Tchebyshev's inequality and a union bound argument, $\forall \eps>0,$ and a constant $C(\eps)\approx \max\paren{26,\frac{96}{\eps^2}}$ such that when  $m\geq C(\eps)n$ we have, 
\begin{align*}
\som{\paren{a_r[1]^4-3}}<\eps,\quad\som{\paren{a_r[1]^2-1}}<\eps,\quad \som{a_r[1]^6\leq20} \\
\qandq \max\limits_{1\leq r\leq m}|a_r[1]|\leq \sqrt{10\log{m}} .
\end{align*}
Each of these event happens with probability at least $1-\frac{1}{n^2}$, and thus their intersection occurs with a probability at least $1-\frac{4}{n^2}$. On this intersection event, we have 
\begin{align*}
V(v)\leq\eps(v[1]^2+\normm{\wtilde{v}}^2)+2\left|\som{a_r[1]^3v[1]\wtilde{a}_r^\top\wtilde{v}}\right|
+\left|\som{a_r[1]^2\paren{\wtilde{a}_r^\top\wtilde{v}-\normm{\tilde{v}}^2}}\right|.
\end{align*}
On the one hand, by a Hoeffding-type inequality (\cite[Proposition~5.10]{vershynin_introduction_2011}), we have 
\begin{align*}
\forall \vrho'>0,\quad \left|\som{a_r[1]^3v[1]\wtilde{a}_r^\top\wtilde{v}}\right|<\vrho'|v[1]|\normm{\wtilde{v}}^2, 
\end{align*}
with a probability $1-ee^{-\zeta' n}\geq 1-3e^{-\zeta' n}$, when $m\geq C(\vrho')\sqrt{n\sum_{r=1}^m a_r[1]^6}$ with $C(\vrho')\approx \frac{1}{\vrho'^2}$ and $\zeta'>2$  an absolute constant.\\
On the other hand, by Bernstein-type inequality (\cite[Proposition~5.16]{vershynin_introduction_2011}), we have 
\begin{align*}
\forall \vrho'>0,\quad \left|\som{a_r[1]^2\paren{\wtilde{a}_r^\top\wtilde{v}-\normm{\wtilde{v}}^2}}\right|\leq\vrho'\normm{\wtilde{v}}^2, 
\end{align*}
with a probability $1-2e^{-\zeta' n}$, when $m\geq C(\vrho')\paren{\sqrt{n\sum_{r=1}^m a_r[1]^4}+n\max\limits_{1\leq r\leq m}a_r[1]^2}$ with $C(\vrho')\approx \frac{1}{\vrho'^2}$.\\
Overall, for any $v\in \bbS^{n-1}$, we have with probability at least $1-5e^{-\zeta' n}$
\[
V(v)\leq \eps+3\vrho'. 
\]
At this stage, we use a covering argument (\cite[Lemma~5.4]{vershynin_introduction_2011}) with an $\frac{1}{2}-$net whose cardinality is smaller than $5^n$. Therefore, choosing $\eps=\vrho'$ and $\vrho=12\vrho'$ we get the claim where $\zeta=\zeta'-\log(5)>0$ since $\zeta'>2$ in the Hoeffding and Bernstein inequalities used above.
\end{proof}

\begin{lemma}\label{lem:conhess_G}\textbf{(Concentration of the Hessian)}
Fix $\vrho\in]0,1[$. If the number of samples obeys $m \geq C(\vrho)n\log n$, for some sufficiently large constant $C(\vrho) > 0$, then 
\begin{align}\label{eq: conhess_G}
\normm{\nabla^2f(x)-\esp{\nabla^2f(x)}}\leq \vrho\paren{\normm{x}^2+\frac{\normm{\avx}^2}{3}}
\end{align}
holds simultaneously for all $x\in\bbR^n$ with a probability at least $1-5e^{-\zeta n}-\frac{4}{n^2},$  where $\zeta$ is a fixed numerical constant. 
\end{lemma}
\begin{proof} 
Recall $\nabla^2f(x)$ from \eqref{hessg}.
By the triangle inequality and Lemma~\ref{lem:expe_lem_G}, we have
\begin{align*}
\normm{\nabla^2f(x)-\esp{\nabla^2f(x)}}\leq &3\normm{\som{|\transp{a_r}x|^2a_r\transp{a_r}}-\paren{2x\transp{x}+\normm{x}^2\Id}}\\
+&\normm{\som{|\transp{a_r}\avx|^2a_r\transp{a_r}-\paren{2\avx\transp{\avx}+\normm{\avx}^2\Id}}}.
\end{align*}
The claim is then a consequence of Lemma~\ref{lem:conhess_G_int}.
\end{proof}

\subsubsection{Injectivity of the measurement operator}
The next result shows that when the number of measurements is large enough, the measurement matrix $A$ (whose rows are the $\transp{a_r}$'s) is injective \whp.
\begin{lemma}\label{pro:injectivity_G} 
Fix $\vrho\in]0,1[$. Assume that $m\geq\frac{16}{\vrho^2}n$. Then 
\begin{equation}\label{eq:injectivite_G}
\paren{1-\vrho}\normm{x}^2 \leq \frac{1}{m}\normm{Ax}^2 \leq (1+\vrho)\normm{x}^2, \quad \forall x\in\bbR^n .
\end{equation}
This happens with a probability at least $1-2e^{-mt^2/2}$ with $\frac{\vrho}{4}=t^2+t$. 
\end{lemma}
\begin{proof} 
This is a consequence of very standard deviation inequalities on the singular values of Gaussian random matrices; see \cite[Lemma~3.1]{candes_phaselift_2013} for a similar statement.
\end{proof}

\subsubsection{Relative smoothness}\label{PrLocalSmad}
For the Gaussian phase retrieval, we have the following refined dimension-independent estimate of the relative smoothness modulus, which is much better that the bound of Proposition~\ref{Tsmad}. 

\begin{lemma}\label{pro:Lsmad_G} 
Fix $\vrho\in]0,1[$. If the event $\calE_{\rm conH}$ defined by \eqref{eq:uniconcen_G} holds true then,
\begin{align}
D_f(x,z)\leq \Ppa{3+\vrho\max\pa{\normm{\avx}^2/3,1}}D_{\psi}(x,z), \qquad \forall x,z \in\bbR^n .  
\end{align}
\end{lemma}

\begin{proof}
Using \eqref{eq:uniconcen_G}, Lemma~\ref{lem:expe_lem_G} and \eqref{hessent}, we have
\begin{align*}
\forall x\in\bbR^n,\quad \nabla^2f(x)&\preceq\esp{\nabla^2f(x)}+\vrho\para{\normm{x}^2+\frac{\normm{\avx}^2}{3}}\Id,\\
&\lon 3\paren{2x\transp{x}+\normm{x}^2\Id} -2\avx\transp{\avx}-\normm{\avx}^2\Id \\
&\quad + \vrho\max\pa{\normm{\avx}^2/3,1}\para{\normm{x}^2+1}\Id,\\
&\lon 3\paren{2x\transp{x}+(\normm{x}^2+1)\Id}+\vrho\max\pa{\normm{\avx}^2/3,1}\nabla^2\psi(x),\\
&= 3\nabla^2\psi(x)+\vrho\max\pa{\normm{\avx}^2/3,1}\nabla^2\psi(x).\label{taylo}\numberthis
\end{align*}
We conclude by applying Lemma~\ref{Bregcomp}.
\end{proof}

\subsubsection{Local relative strong convexity}\label{Localconvex}
The next proposition establishes strong convexity of $f$ \tcb{relative to $\psi$ on a sufficiently small ball around $\overline{\calX}$}. In view of strong $1$-convexity of $\psi$, \tcb{our result also implies strong convexity on the same ball} as shown in \cite{Candes_WF_2015,sun_geometric_2018}.

\begin{lemma}\label{pro:Localconvex_G}
Fix $\lambda \in ]0,1[$ and $\vrho\in]0,\lambda\min\pa{\normm{\avx}^2,1}/(2\max\pa{\normm{\avx}^2/3,1})$. If the event $\calE_{\rm conH}$ defined by \eqref{eq:uniconcen_G} holds true then for all $x,z\in B\BPa{\avx,\frac{1-\lambda}{\sqrt{3}}\normm{\avx}}$ and $x,z\in B\BPa{-\avx,\frac{1-\lambda}{\sqrt{3}}\normm{\avx}}$, 
\begin{align}
D_{f}(x,z) \geq \Ppa{\lambda\min\pa{\normm{\avx}^2,1}-\vrho\max\pa{\normm{\avx}^2/3,1}} D_{\psi}(x,z) .  
\end{align}
\end{lemma}

Observe that if $\normm{\avx}=1$ the above result has a simpler statement. In particular, $\vrho$ must lie in $]0,\lambda[$, and the local relative strong convexity modulus is $\lambda-\vrho$ on a ball of radius $\frac{1-\lambda}{\sqrt{3}}$ around $\overline{\calX}$.

\begin{proof}
We embark from \eqref{eq:uniconcen_G} and Lemma~\ref{lem:expe_lem_G} to infer that $\forall x\in\bbR^n$
\begin{align}
\nabla^2f(x)
&\slon-\vrho\paren{\normm{x}^2+\frac{\normm{\avx}^2}{3}}\Id+3\Ppa{2x\transp{x}+\normm{x}^2\Id}-\paren{2\avx\transp{\avx}+\normm{\avx}^2\Id} \\
&\slon-\vrho\max\pa{\normm{\avx}^2/3,1}\nabla^2\psi(x)+3\Ppa{2x\transp{x}+\normm{x}^2\Id}-\paren{2\avx\transp{\avx}+\normm{\avx}^2\Id} .
\end{align}
We then obtain, for any $v\in\mathbb{S}^{n-1}$
\begin{align*}
\transp{v}\nabla^2f(x)v + \vrho\max\pa{\normm{\avx}^2/3,1}\transp{v}\nabla^2\psi(x)v
&\geq 3\Ppa{2\paren{\transp{v}x}^2+\normm{x}^2}-\Ppa{2\paren{\transp{v}\avx}^2+\normm{\avx}^2} .
\end{align*}
Let $\rho>0$ small enough, to be made precise later. Thus for any $x=\pm\avx+\rho v$ we get
\begin{align*}
&\transp{v}\nabla^2f(x)v + \vrho\max\pa{\normm{\avx}^2/3,1}\transp{v}\nabla^2\psi(x)v\\
&\geq 6\paren{\transp{v}\avx}^2+6\rho^2\pm12\rho\transp{v}\avx+3\normm{\avx}^2\pm6\rho\transp{v}\avx+3\rho^2-2\paren{\transp{v}\avx}^2-\normm{\avx}^2 \\
&= 4\paren{\transp{v}\avx}^2+9\rho^2\pm18\rho\transp{v}\avx+2\normm{\avx}^2 .
\end{align*}
From \eqref{hessent}, we also have 
\begin{align*}
\transp{v}\nabla^2\psi(x)v&=\normm{x}^2+1+2\paren{\transp{v}x}^2= 2\paren{\transp{v}\avx}^2 +3 \rho^2 + \pm6\rho\transp{v}\avx + \normm{\avx}^2 + 1 . 
\end{align*}
Consider first the case where $\normm{\avx} \geq 1$. We then get
\begin{align*}
&\transp{v}\para{\nabla^2f(x)-\Ppa{\lambda-\vrho\max\pa{\normm{\avx}^2/3,1}}\nabla^2\psi(x)}v \\
&\geq 2(2-\lambda)\paren{\transp{v}\avx}^2+3(3-\lambda)\rho^2 \pm 6(3-\lambda)\rho\transp{v}\avx+(2-\lambda)\normm{\avx}^2-\lambda \\
&= 2(2-\lambda)\paren{\transp{v}\avx}^2+3(3-\lambda)\rho^2 \pm 6(3-\lambda)\rho\transp{v}\avx+2(1-\lambda)\normm{\avx}^2+\lambda(\normm{\avx}^2-1) \\
&\geq 2(2-\lambda)\paren{\transp{v}\avx}^2+3(3-\lambda)\rho^2 \pm 6(3-\lambda)\rho\transp{v}\avx+2(1-\lambda)\normm{\avx}^2 .
\end{align*}
We claim that
\[
\inf_{v \in \mathbb{S}^{n-1}} 2(2-\lambda)\paren{\transp{v}\avx}^2+3(3-\lambda)\rho^2 \pm 6(3-\lambda)\rho\transp{v}\avx+2(1-\lambda)\normm{\avx}^2 \geq 0
\]
for $\rho$ small enough. Let $\transp{v}\avx = \alpha\normm{\avx}$, where $\alpha \in [-1,1]$ and $\rho = \beta \normm{\avx}$. Thus
\begin{multline*}
2(2-\lambda)\paren{\transp{v}\avx}^2+3(3-\lambda)\rho^2 \pm 6(3-\lambda)\rho\transp{v}\avx+2(1-\lambda)\normm{\avx}^2 = \\
\Ppa{2(2-\lambda)\alpha^2 + 3(3-\lambda)\beta^2 \pm 6(3-\lambda)\alpha\beta + 2(1-\lambda)}\normm{\avx}^2 .
\end{multline*}
Minimizing the last term for $\alpha$ and substituting back, we have after simple algebra that
\begin{align*}
2(2-\lambda)\alpha^2 + 3(3-\lambda)\beta^2 \pm 6(3-\lambda)\alpha\beta + 2(1-\lambda) \geq 2(1-\lambda) - \phi(\lambda)\beta^2 .
\end{align*}
where we set the function $\phi: t \in ]0,1[ \mapsto \frac{36(3-t)^2}{8(2-t)}-3(3-t) \in \bbR_+$. It can be easily shown that $\sup_{]0,1[}\phi(t) = \phi(1)=12$. In turn, we have
\begin{align*}
2(2-\lambda)\alpha^2 + 3(3-\lambda)\beta^2 \pm 6(3-\lambda)\alpha\beta + 2(1-\lambda) \geq 0
\end{align*}
since we assumed that $\rho \leq \frac{1-\lambda}{\sqrt{3}}\normm{\avx} \leq \frac{2(1-\lambda)}{\sqrt{\phi(\lambda)}}\normm{\avx}$.\\

Let us now turn to the case where $\normm{\avx} \leq 1$. We then have
\begin{align*}
&\transp{v}\para{\nabla^2f(x)-\lambda\Ppa{\normm{\avx}^2-\vrho\max\pa{\normm{\avx}^2/3,1}}\nabla^2\psi(x)}v \\
&\geq 2(2-\lambda\normm{\avx}^2)\paren{\transp{v}\avx}^2+3(3-\lambda\normm{\avx}^2)\rho^2 \pm 6(3-\lambda\normm{\avx}^2)\rho\transp{v}\avx+2\normm{\avx}^2-\lambda\normm{\avx}^4-\lambda\normm{\avx}^2 \\
&\geq 2(2-\lambda\normm{\avx}^2)\paren{\transp{v}\avx}^2+3(3-\lambda\normm{\avx}^2)\rho^2 \pm 6(3-\lambda\normm{\avx}^2)\rho\transp{v}\avx+2(1-\lambda)\normm{\avx}^2+\lambda\Ppa{\normm{\avx}^2-\normm{\avx}^4} \\
&\geq 2(2-\lambda\normm{\avx}^2)\paren{\transp{v}\avx}^2+3(3-\lambda\normm{\avx}^2)\rho^2 \pm 6(3-\lambda\normm{\avx}^2)\rho\transp{v}\avx+2(1-\lambda)\normm{\avx}^2 \\
&= \Ppa{2(2-\lambda\normm{\avx}^2)\alpha^2+3(3-\lambda\normm{\avx}^2)\beta^2 \pm 6(3-\lambda\normm{\avx}^2)\alpha\beta+2(1-\lambda)}\normm{\avx}^2 .
\end{align*}
Arguing as in the first case, we have
\[
2(2-\lambda\normm{\avx}^2)\alpha^2+3(3-\lambda\normm{\avx}^2)\beta^2 \pm 6(3-\lambda\normm{\avx}^2)\alpha\beta+2(1-\lambda) \geq 2(1-\lambda) - \phi(\lambda\normm{\avx}^2)\beta^2 .
\]
Thus, the right hand side is non-negative since 
\[
\beta \leq \frac{1-\lambda}{\sqrt{3}} \leq \frac{2(1-\lambda)}{\sqrt{\phi(\lambda\normm{\avx}^2)}} ,
\]
where we used that $\normm{\avx}^2 \leq 1$ in the argument of $\phi$.
\\
Overall, we have shown that 
\[
\transp{v}\para{\nabla^2f(x)-\Ppa{\lambda\min\pa{\normm{\avx}^2,1}/2-\vrho\max\pa{\normm{\avx}^2/3,1}}\nabla^2\psi(x)}v \geq 0
\]
for all $v \in \mathbb{S}^{n-1}$ and $\rho \leq \frac{1-\lambda}{\sqrt{3}}\normm{\avx}$. We complete the proof by invoking Lemma~\ref{Bregcomp} and convexity of the ball.
\end{proof}

\subsubsection{Spectral initialization}\label{Prpro:spectralinit}
We now show that the initial guess $\xo$ generated by spectral initialization (Algorithm~\ref{alg:algoSPIniit}) belongs to a small $f$-attentive neighborhood of $\overline{\calX}$.
\begin{lemma}\label{pro:spectralinit}
Fix $\vrho\in ]0,1[$. If the number of samples obeys $m \geq C(\vrho)n\log n$, for some sufficiently large constant $C(\vrho) > 0$, then with probability at least $1-2e^{-\frac{m\pa{\sqrt{1+\vrho}-1}^2}{8}}-5e^{-\zeta n}-\frac{4}{n^2}$,  where $\zeta$ is a fixed numerical constant, $\xo$ satisfies:
\begin{enumerate}[label=(\roman*)]
\setlength{\itemindent}{0.4cm}
\item \label{pro:spectralinit-i} 
$\dist(\xo,\overline{\calX}) \leq \eta_1(\vrho)\normm{\avx}$, 
where 
\begin{alignat}{2}
\eta_1\colon ]0,1[ & \to {}&& ]0,1[ \nonumber\\
		\vrho & \mapsto {}&& \Ppa{\sqrt{2-2\sqrt{1-\vrho}} + \vrho/2} , \label{radius}
\end{alignat}
which is an increasing function.

\item \label{pro:spectralinit-ii} $f(\xo)\leq\Ppa{3+\vrho\max\pa{\normm{\avx}^2/3,1}}\frac{\Theta(\eta_1(\vrho)\normm{\avx})}{2}\eta_1(\vrho)^2\normm{\avx}^2$.
        
\item \label{pro:spectralinit-iii} Besides, for $\lambda \in ]0,1[$, if 
\begin{equation}\label{eq:spectralinitvrhobnd}
\vrho\leq\eta_1^{-1}\Ppa{\frac{1-\lambda}{\sqrt{3\Ppa{6(1+(1-\lambda)^2/3)+1}}}\frac{1}{\max\Ppa{\normm{\avx},1}}} ,
\end{equation} 
then with the same probability as above $\xo \in B\Ppa{\overline{\calX},\frac{\rho}{\max\Ppa{\sqrt{\Theta(\rho)},1}}}$
where $\rho=\frac{1-\lambda}{\sqrt{3}}\normm{\avx}$. 
\end{enumerate}
\end{lemma}


\begin{proof}{\ }
\begin{enumerate}[label=(\roman*)]
  \setlength{\itemindent}{0.4cm}
\item 
Denote the matrix
\[
Y=\som{y[r]a_r\transp{a_r}}=\som{|\transp{a_r}\avx|^2a_r\transp{a_r}}.
\]
By Lemma~\ref{lem:conhess_G_int}, we have \whp 
\[
\normm{Y-\esp{Y}}\leq\vrho \normm{\avx}^2.
\]
Let $\wtilde{x}$ be the eigenvector associated with the largest eigenvalue $\wtilde{\lambda}$ of $Y$ such that $\normm{\wtilde{x}}=\normm{\avx}$ (obviously $\wtilde{\lambda}$ is nonnegative since $Y$ is semidefinite positive). Then, 
\begin{align*}
\vrho \normm{\avx}^2
\geq\normm{Y-\mathbb{E}(Y)}
&\geq \normm{\avx}^{-2}\abs{\transp{\wtilde{x}}\paren{Y-2\avx\transp{\avx}-\normm{\avx}^2\Id}\wtilde{x}} \\
&= \normm{\avx}^{-2}\abs{\wtilde{\lambda}\normm{\avx}^2 -2\pa{\transp{\wtilde{x}}\avx}^2-\normm{\avx}^4} .
\end{align*}
Hence 
\[
2\pa{\transp{\wtilde{x}}\avx}^2 \geq \wtilde{\lambda}\normm{\avx}^2  - (1 + \vrho)\normm{\avx}^4 .
\]
Moreover, using Lemma~\ref{lem:conhess_G_int} again entails that \whp
\[
\wtilde{\lambda}\normm{\avx}^2 \geq \transp{\avx}Y\avx \geq \transp{\avx}\paren{2\avx\transp{\avx}+\normm{\avx}^2\Id}\avx - \vrho\normm{\avx}^4 = (3 - \vrho)\normm{\avx}^4 .
\]
Combining the last two inequalities, we get
\[
\pa{\transp{\wtilde{x}}\avx}^2 \geq (1-\rho)\normm{\avx}^4 .
\]
It then follows that
\[
\dist(\wtilde{x},\overline{\calX}) \leq \sqrt{2 - 2\sqrt{1-\rho}}\normm{\avx} .
\]
By definition of $\xo$ in Algorithm~\ref{alg:algoSPIniit}, $\xo=\sqrt{m^{-1}\sum_r y[r]} \frac{\wtilde{x}}{\normm{\avx}}$, and thus \whp
\[
\normm{\xo-\wtilde{x}} = \abs{\sqrt{\frac{m^{-1}\sum_r y[r]}{\normm{\avx}^2}} - 1}\normm{\avx} = \abs{\sqrt{\frac{m^{-1}\normm{A\avx}^2}{\normm{\avx}^2}} - 1}\normm{\avx} \leq \vrho/2\normm{\avx} ,
\]
where we used Lemma~\ref{pro:injectivity_G}. In turn,
\[
\dist(\xo,\overline{\calX}) \leq \dist(\wtilde{x},\overline{\calX}) + \normm{\xo-\wtilde{x}} \leq \Ppa{\sqrt{2-2\sqrt{1-\vrho}} + \vrho/2}\normm{\avx} .
\]

\item Under our sampling complexity bound, event $\calE_{\rm conH}$ defined by \eqref{eq:uniconcen_G} holds true \whp. It then follows from Lemma~\ref{pro:Lsmad_G} applied at $\avx$ and $\xo$, that
\begin{align}
D_{f}(\xo,\avx) \leq \Ppa{3+\vrho\max\pa{\normm{\avx}^2/3,1}}D_{\psi}(\xo,\avx) .
\end{align}
Since $f(\avx)=0$ and $\nabla f(\avx)=0$, we obtain from Proposition~\ref{pp:bregman}\ref{pp:bregman4} that
\begin{align}
f(\xo) 
&\leq \Ppa{3+\vrho\max\pa{\normm{\avx}^2/3,1}}D_{\psi}(\xo,\avx) \nonumber\\
&\leq \Ppa{3+\vrho\max\pa{\normm{\avx}^2/3,1}}\frac{\Theta(\eta_1(\vrho)\normm{\avx})}{2}\eta_1(\vrho)^2\normm{\avx}^2 .
\end{align}

\item In view of \ref{pro:spectralinit-i}, it is sufficient to show that $\eta_1(\vrho)\normm{\avx} \leq \frac{\rho}{\max\Ppa{\sqrt{\Theta(\rho)},1}}$. Since from Proposition~\ref{pp:bregman}\ref{pp:bregman4} (see also Remark~\ref{rem:deterministic}) we have 
\[
\Theta(\rho) \leq 6(\normm{\avx}^2+\rho^2) + 1 \leq \Ppa{6(1+(1-\lambda)^2/3) + 1}\max\Ppa{\normm{\avx}^2,1} ,
\]
and $\eta_1$ is an increasing function, we conclude.
\end{enumerate}
\end{proof}

\subsection{Proofs for the CDP model}
In this section, we assume that the sensing vectors $(a_r)_{r\in \tcb{\bbrac{m}}}$ follow the CDP model introduced in  Section~\ref{framework}.

\subsubsection{Expectation and deviation of the Hessian}
\begin{lemma}\label{lem:expe_lem_cdp}\textbf{(Expectation of the Hessian)}{\ }
 Under the CDP measurement model, the following holds 
  \begin{align}\label{eq:expe_lem_cdp}
   \esp{\nabla^2 f(x)}=3\para{x\transp{x}+ \normm{x}^2\Id}-\avx\transp{\avx}-\normm{\avx}^2\Id. 
   \end{align}
\end{lemma}

\begin{proof}


From \cite[Lemma~3.1]{candes_phase_2015}, we have
\begin{align}
 \forall x\in \bbR^n,\quad \esp{\frac{1}{nP}\sum_{j,p=1}^{n,P}|\adj{f}_jD_px|^2D_pf_j\adj{f}_jD_p}=x\transp{x}+\normm{x}^2\Id.
\end{align}
Combining this with \eqref{hessg} yields the claim.
\end{proof}

Unlike the Gaussian model, it turns out that it is very challenging to concentrate the Hessian of $f$ around its mean simultaneously for all vectors $x \in \bbR^n$ with non-trivial sampling complexity bounds. The main reason is that the CDP model does not have enough randomness to be used in the mathematical analysis. However, one can still do that for a fixed vector $x$. The next lemma gives the Hessian deviation at $\pm\avx$.

\begin{lemma}\label{lem:conhess_cdp}\textbf{(Concentration of the Hessian)}
Fix  $\delta \in]0,1[$. If the number of patterns obeys $P\geq C(\delta)\log^3(n)$, then with a probability at least $1-\frac{4P+1}{2n^3}$
\begin{align}\label{eq:uniconcen_cdp}
\normm{\nabla^2f(\avx)-\esp{\nabla^2f(\avx)}}\leq \delta\normm{\avx}^2 .
\end{align}
\end{lemma}
\begin{proof}
Let $\adj{f}_j$ be the rows of the discrete Fourier transform, \ie $f_j[\ell] = e^{i\frac{2\pi j\ell}{n}}$. With a slight adaptation to the real case of the argument in \cite[Section~A.4.1]{Candes_WF_2015}, we deduce that
\begin{align}\label{eq:conccandes}
\normm{\frac{1}{nP}\sum_{j,p}^{n,P}|\adj{f}_jD_p\avx|^2D_pf_j\adj{f}_jD_p- \paren{\avx\transp{\avx}+\normm{\avx}^2\Id}}\leq \frac{\delta}{2}\normm{\avx}^2,
\end{align}
provided that $P\geq C(\delta)\log^3(n)$ with a probability at least $1-\frac{4P+1}{2n^3}$. Combining this with Lemma~\ref{lem:expe_lem_cdp}, we conclude.
\end{proof}

\subsubsection{Injectivity of the measurement operator}
We now establish that for $m$ large enough, the measurement matrix $A$ is injective \whp. Recall that the rows of $A$ are the $a_r^*$'s.

\begin{lemma}\label{pro:injectivity_cdp} 
Fix $\vrho\in]0,1[$. Assume that $P\geq C(\vrho)\log(n)$. Then with a probability at least $1-1/n^2$
\begin{equation}\label{eq:injectivite_cdp}
\paren{1-\vrho}\normm{x}^2 \leq \frac{1}{m}\normm{Ax}^2 \leq (1+\vrho)\normm{x}^2, \quad \forall x\in\bbR^n .
\end{equation}
\end{lemma}
\begin{proof} 
This is a consequence of the fact that
\[
\normm{\frac{1}{m} A^*A - \Id} \leq \vrho
\]
with the claimed probability. Indeed, as for \cite[Lemma~3.3]{candes_phase_2015}, the covariance matrix $\frac{1}{m} A^*A$ is diagonal with \iid diagonal entries whose expectation is $\esp{d^2}=1$, and the statement follows from Hoeffding's inequality and a union bound. 
\end{proof}

\subsubsection{Local relative smoothness and relative strong convexity}\label{PrLocalSmad_cdp}
We now turn to proving local relative smoothness and relative strong convexity near the true vectors. Unlike the Gaussian case, we only have a local version of relative smoothness. The reason behind this, as discussed above, is that it seems very hard to have a uniform concentration bound for the Hessian of $f$ around its mean for the CDP model. To circumvent this, we use a continuity argument.
 
\begin{lemma}\label{pro:Lsmad_cdp} 
Fix $\delta\in]0,\min(\normm{\avx}^2,1)/2[$. Suppose that \eqref{eq:uniconcen_cdp} holds. Then there exists $\rho_{\delta} > 0$ such that for all $x,z\in B\BPa{\avx,\rho_{\delta}}$ and $x,z\in B\BPa{-\avx,\rho_{\delta}}$
\begin{align}\label{eq:Lsmad_cdp}
\frac{\Ppa{\min\pa{\normm{\avx}^2,1}-2\delta}}{1+\delta}D_{\psi}(x,z) \leq D_f(x,z) \leq 2(1+\delta)^2 D_{\psi}(x,z) .
\end{align}
\end{lemma}
Observe that while in the Gaussian case, the ball radius on which relative strong convexity holds is fixed and explicit, for the CDP model, we only know it exists and it depends on $\delta$. 

\begin{proof}
We prove the claim for $\avx$ and the same holds obviously around $-\avx$.
Using \eqref{eq:uniconcen_cdp} and \eqref{hessent} gives  
\begin{align*}
\nabla^2f(\avx)
 &\preceq\esp{\nabla^2f(\avx)}+\delta\normm{\avx}^2\Id \\
 &= 2\avx\transp{\avx}+(2+\delta)\normm{\avx}^2\Id \\
 &\preceq (2 +\delta)\para{2\avx\transp{\avx}+\pa{\normm{\avx}^2+1}\Id},\\
 &= (2+\delta)\nabla^2\psi(\avx). 
\end{align*}
Again, from \eqref{eq:uniconcen_cdp} and \eqref{hessent}, we get
\begin{align*}
\nabla^2f(\avx)&\slon\esp{\nabla^2f(\avx)}-\delta\normm{\avx}^2\Id \slon 2(\avx\transp{\avx}+\normm{\avx}^2\Id) -\delta\nabla^2\psi(\avx)\Id.
\end{align*}
If $\normm{\avx}\geq1$, we arrive at
\begin{align*}
\nabla^2f(\avx)
\slon 2\avx\transp{\avx}+ (\normm{\avx}^2+1)\Id-\delta\nabla^2\psi(\avx)\Id = (1-\delta)\nabla^2\psi(\avx).
\end{align*}
If $\normm{\avx} \leq 1$, we have
\begin{align*}
\nabla^2f(\avx)-\para{\normm{\avx}^2-\delta}\nabla^2\psi(\avx)
&\slon 2\avx\transp{\avx}+2\normm{\avx}^2\Id -2\normm{\avx}^2\avx\transp{\avx}-\normm{\avx}^4\Id-\normm{\avx}^2\Id \\
&= 2(1-\normm{\avx}^2)\avx\transp{\avx} + (\normm{\avx}^2 - \normm{\avx}^4)\Id \slon 0 .
\end{align*}
Therefore
\begin{align}\label{eq:tmp1}
\Ppa{\min\pa{\normm{\avx}^2,1}-\delta}\nabla^2\psi(\avx) \preceq \nabla^2 f(\avx) \preceq (2+\delta)\nabla^2\psi(\avx) . 
\end{align}
Combining \eqref{eq:tmp1} with continuity of $\nabla^2 f$ and $1$-strong convexity of $\psi$, $\exists \rho_{\delta} > 0$ such that $\forall x\in B(\avx,\rho_{\delta})$ we have  
\begin{equation}\label{eq:bndhessfpsi_cdp}
\Ppa{\min\pa{\normm{\avx}^2,1}-2\delta}\nabla^2\psi(\avx) \preceq \nabla^2 f(\avx)-\delta\Id \preceq \nabla^2f(x) \preceq \nabla^2f(\avx)+\delta\Id \preceq 2(1+\delta)\nabla^2\psi(\avx) .
\end{equation}
Continuity of $\nabla^2 \psi$ and $1$-strong convexity of $\psi$ also yield that $\forall x\in B(\avx,\rho_{\delta})$
\begin{equation}\label{eq:conthesspsi_cdp}
\nabla^2\psi(\avx) \preceq \nabla^2\psi(x)+\delta \Id \preceq (1+\delta)\nabla^2\psi(x) \qandq \nabla^2\psi(x) \preceq \nabla^2\psi(\avx) + \delta\Id \preceq (1+\delta)\nabla^2\psi(\avx) .
\end{equation}
Combining \eqref{eq:bndhessfpsi_cdp} and \eqref{eq:conthesspsi_cdp}, we obtain that $\forall x \in B(\avx,\rho_{\delta})$,
\begin{align*}
\frac{\Ppa{\min\pa{\normm{\avx}^2,1}-2\delta}}{1+\delta}\nabla^2\psi(x) \preceq \nabla^2f(x) \preceq 2(1+\delta)^2\nabla^2\psi(x) . 
\end{align*}
Invoking Lemma~\ref{Bregcomp} and convexity of the ball, we get the statement.
\end{proof}

\subsubsection{Spectral initialization}
We now show the analogue of Lemma~\ref{pro:spectralinit} for the CDP measurement model.

\begin{lemma}\label{pro:spectralinit_cdp}
Fix $\vrho\in ]0,1[$. If the number of patterns obeys $P \geq C(\vrho)n\log^3(n)$, for some sufficiently large constant $C(\vrho) > 0$, then with probability at least $1-\frac{4P+1}{n^3}-\frac{1}{n^2}$, $\xo$ satisfies:
\begin{enumerate}[label=(\roman*)]
 \setlength{\itemindent}{0.4cm}
\item \label{pro:spectralinit_cdp-i} 
$\dist(\xo,\overline{\calX}) \leq \eta_1(\vrho)\normm{\avx}$, 
where $\eta_1$ is the function defined in \eqref{radius}.\\

Let $\delta\in]0,\min(\normm{\avx}^2,1)/2[$ and $\rho_{\delta}$ is the neighborhood radius in Lemma~\ref{pro:Lsmad_cdp}. Suppose that $\vrho$ is sufficiently small, \ie
\begin{equation}\label{eq:spectralinitvrhobnd_cdp}
\vrho \leq \min\Ppa{\delta,\eta_1^{-1}\Ppa{\frac{\rho_{\delta}/\normm{\avx}}{\sqrt{6(\normm{\avx}^2+\rho_{\delta}^2)+1}}}} .
\end{equation}
Then, with the same probability as above,
\item \label{pro:spectralinit_cdp-ii} $f(\xo)\leq 2(1+\delta^2)\frac{\Theta(\eta_1(\vrho)\normm{\avx})}{2}\eta_1(\vrho)^2\normm{\avx}^2$ ;
        
\item \label{pro:spectralinit_cdp-iii} $\xo \in B\Ppa{\overline{\calX},\frac{\rho_{\delta}}{\max\Ppa{\sqrt{\Theta(\rho_{\delta})},1}}}$. 
\end{enumerate}
\end{lemma}
\begin{proof}
\begin{enumerate}[label=(\roman*)]
The proof of this claim is similar to that of Lemma~\ref{pro:spectralinit} for the Gaussian case, where we now invoke Lemma~\ref{lem:conhess_cdp} and Lemma~\ref{pro:injectivity_cdp} for statement~\ref{pro:spectralinit_cdp-i}. For the last two claims, we also use Lemma~\ref{pro:Lsmad_cdp} and that $\vrho$ is small enough as prescribed. 
\end{enumerate}
\end{proof}
}

\end{appendices}

\begin{acknowledgments}
The authors thank the French National Research Agency (ANR) for funding the project FIRST (ANR-19-CE42-0009). 
\end{acknowledgments}

\small
\bibliographystyle{plain}
{
\bibliography{biblio/biblio3}
}

\end{document}